\documentclass[11pt]{amsart}

\usepackage{enumerate}
\usepackage{amsmath,amsthm,verbatim,amssymb,amsfonts,amscd,amsopn,amsxtra,graphicx,lmodern}
\usepackage{hyperref}
\usepackage{tikz-cd} 
\usepackage{graphics}
\usepackage{mathrsfs}
\usepackage{relsize}
\usepackage{enumitem}
\usepackage[utf8]{inputenc}
\usepackage{csquotes}
\usepackage{amsthm}
\usepackage{thmtools}
\usepackage{mathtools}
\usepackage{microtype}
\usepackage{pdfrender,xcolor}
\usepackage[edges]{forest}
\usepackage{qtree, makecell, caption}
\usepackage[sort cites=true, backend=biber]{biblatex}
\usepackage{biblatex}

\tikzset{
	symbol/.style={
		draw=none,
		every to/.append style={
			edge node={node [sloped, allow upside down, auto=false]{$#1$}}}
	}
}

\tikzset{
	solid node/.style={circle,draw,inner sep=1.5,fill=black},
	hollow node/.style={circle,draw,inner sep=1.5},
	rectangle node/.style={rectangle,draw,align=center}
}

\begin{document}
\pdfrender{StrokeColor=black,TextRenderingMode=2,LineWidth=0.2pt}	
	
	\title{On the ranks and implicit constant fields of valuations induced by pseudo monotone sequences}
	\author{Arpan Dutta}
	\def\NZQ{\mathbb}               
	\def\NN{{\NZQ N}}
	\def\QQ{{\NZQ Q}}
	\def\ZZ{{\NZQ Z}}
	\def\RR{{\NZQ R}}
	\def\CC{{\NZQ C}}
	\def\AA{{\NZQ A}}
	\def\BB{{\NZQ B}}
	\def\PP{{\NZQ P}}
	\def\FF{{\NZQ F}}
	\def\GG{{\NZQ G}}
	\def\HH{{\NZQ H}}
	\def\UU{{\NZQ U}}
	\def\P{\mathcal P}
	
	%
	%
	\let\union=\cup
	\let\sect=\cap
	\let\dirsum=\oplus
	\let\tensor=\otimes
	\let\iso=\cong
	\let\Union=\bigcup
	\let\Sect=\bigcap
	\let\Dirsum=\bigoplus
	\let\Tensor=\bigotimes
	
	\theoremstyle{plain}
	\newtheorem{Theorem}{Theorem}[section]
	\newtheorem{Lemma}[Theorem]{Lemma}
	\newtheorem{Corollary}[Theorem]{Corollary}
	\newtheorem{Proposition}[Theorem]{Proposition}
	\newtheorem{Problem}[Theorem]{}
	\newtheorem{Conjecture}[Theorem]{Conjecture}
	\newtheorem{Question}[Theorem]{Question}
	
	\theoremstyle{definition}
	\newtheorem{Example}[Theorem]{Example}
	\newtheorem{Examples}[Theorem]{Examples}
	\newtheorem{Definition}[Theorem]{Definition}
	
	\theoremstyle{remark}
	\newtheorem{Remark}[Theorem]{Remark}
	\newtheorem{Remarks}[Theorem]{Remarks}
	
	\newcommand{\trdeg}{\mbox{\rm trdeg}\,}
	\newcommand{\rr}{\mbox{\rm rat rk}\,}
	\newcommand{\rk}{\mbox{\rm rk}\,}
	\newcommand{\sep}{\mathrm{sep}}
	\newcommand{\ac}{\mathrm{ac}}
	\newcommand{\ins}{\mathrm{ins}}
	\newcommand{\res}{\mathrm{res}}
	\newcommand{\Gal}{\mathrm{Gal}\,}
	\newcommand{\ch}{\mathrm{char}\,}
	\newcommand{\Aut}{\mathrm{Aut}\,}
	\newcommand{\kras}{\mathrm{kras}\,}
	\newcommand{\dist}{\mathrm{dist}\,}
	\newcommand{\ord}{\mathrm{ord}\,}
	\newcommand{\dom}{\mathrm{dom}\,}
	\newcommand{\lex}{\mathrm{lex}\,}

	\newcommand{\n}{\par\noindent}
	\newcommand{\nn}{\par\vskip2pt\noindent}
	\newcommand{\sn}{\par\smallskip\noindent}
	\newcommand{\mn}{\par\medskip\noindent}
	\newcommand{\bn}{\par\bigskip\noindent}
	\newcommand{\pars}{\par\smallskip}
	\newcommand{\parm}{\par\medskip}
	\newcommand{\parb}{\par\bigskip}
	\let\epsilon\varepsilon
	\let\phi=\varphi
	\let\kappa=\varkappa
	
	\def \a {\alpha}
	\def \b {\beta}
	\def \s {\sigma}
	\def \d {\delta}
	\def \g {\gamma}
	\def \o {\omega}
	\def \l {\lambda}
	\def \th {\theta}
	\def \D {\Delta}
	\def \G {\Gamma}
	\def \O {\Omega}
	\def \L {\Lambda}
	%
	%
	\textwidth=15cm \textheight=22cm \topmargin=0.5cm
	\oddsidemargin=0.5cm \evensidemargin=0.5cm \pagestyle{plain}

	\address{Department of Mathematics, IISER Mohali,
		Knowledge City, Sector 81, Manauli PO,
		SAS Nagar, Punjab, India, 140306.}
	\email{arpan.cmi@gmail.com}
	
	\date{\today}
	
	\thanks{}
	
	
	
	\maketitle
	

\begin{abstract}
Given a valued field $(K,v)$ and a pseudo monotone sequence $E$ in $(K,v)$, one has an induced valuation $v_E$ extending $v$ to $K(X)$. After fixing an extension of $v_E$ to a fixed algebraic closure $\overline{K(X)}$ of $K(X)$, we show that the implicit constant field of the extension $(K(X)|K,v_E)$ is simply the henselization of $(K,v)$. We consider the question: given a value transcendental extension $w$ of $v$ to $K(X)$ and a pseudo monotone sequence $E$ in $(K,v)$, under which precise conditions is $w$ induced by $E$? The dual nature of pseudo convergent sequences of algebraic type and pseudo divergent sequences is also explored. Further, we provide a complete description of the various possibilities of the rank of the valuation $v_E$, provided that $v$ has finite rank. 
\end{abstract}

\section{Introduction}
A central problem in valuation theory is to study the set of all extensions of a valuation $v$ from a field $K$ to a rational function field $K(X)$. On the other hand, given such an extension $w$ of $v$ to $K(X)$, it is important to give a complete description of the valuation $w$. Several different objects have been introduced to tackle these interconnected problems. Minimal pairs of definition were introduced by Alexandru, Popescu and Zaharescu [\ref{AP sur une classe}, \ref{APZ characterization of residual trans extns}] to study residue transcendental extensions. The notion of key polynomials was introduced by MacLane [\ref{MacLane key pols}] which was later generalized by Vaqui\'{e} [\ref{Vaquie key pols}]. An alternative form of key polynomials have been recently developed by Spivakovsky and Novacoski [\ref{Nova Spiva key pol pseudo convergent}, \ref{Novacoski key poly and min pairs}]. Ostrowski [\ref{Ostrowski pcs}] introduced pseudo convergent sequences which were succesfully used by Kaplansky [\ref{Kaplansky}] to answer maximality questions for immediate extensions. Recently, this notion has been generalized to that of pseudo monotone sequences which has found significant use in the works of Peruginelli and Spirito [\ref{Peruginelli, Spirito - extend valns pseudo monotone}]. The notion of implicit constant fields has been introduced by Kuhlmann [\ref{Kuh value groups residue fields rational fn fields}].

\pars Our goal in this article is to study the extensions of a valued field $(K,v)$ induced by pseudo monotone sequences in $(K,v)$. In particular, given an extension $v_E$ of $v$ to $K(X)$ induced by a pseudo monotone sequence $E$ in $(K,v)$ and an extension of $v_E$ to a fixed algebraic closure $\overline{K(X)}$ of $K(X)$, we investigate the following:
\begin{itemize}
	\item the implicit constant field of the extension $(K(X)|K,v_E)$,
	\item the rank of the value group $v_E K(X)$.
\end{itemize}
It has recently been observed by Peruginelli and Spirito [\ref{Peruginelli, Spirito - extend valns pseudo monotone}] that every extension of a valuation $v$ from $K$ to $K(X)$ is induced by a pseudo monotone sequence in $(K,v)$ if and only if the $v$-adic completion $\hat{K}$ is algebraically closed. As a consequence, our results hold for all extensions $(K(X)|K,w)$ whenever $\hat{K}$ is algebraically closed. 

\pars Given a valued field extension $(K(X)|K,w)$ and an extension of $w$ to $\overline{K(X)}$, the \textbf{implicit constant field} of the extension is defined to be the relative algebraic closure of $K$ in the henselization of $K(X)$ and it is denoted by $IC(K(X)|K,w)$. In particular, $IC(K(X)|K,w)$ is a separable-algebraic extension of the henselization $K^h$ of $(K,v)$. It has been observed in [\ref{Dutta min fields implicit const fields}, Theorem 1.1 and Theorem 1.3] that explicit computation of $IC(K(X)|K,w)$ can be quite difficult in general. In this article we discover that
\[ IC(K(X)|K,v_E)= K^h \]
whenever $E$ is a pseudo monotone sequence in $(K,v)$ (see Theorem \ref{Thm K(X)|K is pure} and Theorem \ref{Thm pcs of alg type}). 

\pars The extension $(K(X)|K,v_E)$ is value transcendental if and only if $E$ is either pseudo convergent of algebraic type or pseudo divergent. We consider the question:
\begin{Question}\label{question}
	Given a value transcendental extension $(K(X)|K,w)$ and a pseudo monotone sequence $E$ in $(K,v)$, under which precise conditions is $w = v_E$?
\end{Question}

Some related equivalences for this equality to hold are provided in Theorem \ref{Thm pcs alg equivalences} and Theorem \ref{Thm E pds}. In particular, we observe the strong connections between pairs of definition for $v_E$ over $K$ and the limits of $E$, which play a crucial role throughout this paper. The duality of pseudo convergent sequences of algebraic type and pseudo divergent sequences is made clear in Theorem \ref{Thm pcs alg equivalences}$(iii)$ and Theorem \ref{Thm E pds}$(iii)$. In the corollary to Theorem \ref{Thm pcs alg equivalences} we observe that if $E$ is a Cauchy sequence and $(K(X)|K,w)$ is a value transcendental extension such that $X$ is a limit of $E$, then $w=v_E$. Similarly, if $E$ diverges to infinity and $X$ is not a limit of $E$, then $w=v_E$.   
 
\pars We also provide a complete characterization of the rank of the value group $v_E K(X)$, provided that $vK$ is of finite rank. It is well known that the rank either remains same or increases by one. We give a constructive description of the various possibilities (see Figures \ref{fig E pcs base level}, \ref{fig E pcs i level}, \ref{fig E pds base level}, \ref{fig E pds i level}). Since the value group remains unchanged when $E$ is a pseudo convergent sequence of transcendental type or a pseudo constant sequence, we only focus on the cases when $E$ is pseudo convergent of algebraic type or pseudo divergent. The duality observed in Theorem \ref{Thm pcs alg equivalences}$(iii)$ and Theorem \ref{Thm E pds}$(iii)$ comes to the fore at this juncture. 

\pars After fixing an extension $\overline{v}$ of $v$ to $\overline{K}$, we introduce the notion of supremum and infimum of a pseudo convergent sequence of algebraic type or a pseudo divergent sequence, when $(K,v)$ has finite rank. Thereafter, we observe in Theorem \ref{Thm rank v_E} that whenever $E$ is a Cauchy sequence or a pseudo convergent sequence of algebraic type such that the supremum of $E$ is contained in $\overline{v}\overline{K}$, we obtain $\rk v_E K(X) = \rk vK +1$. The analogous result holds true when $E$ diverges to infinity or $E$ is pseudo divergent such that the infimum of $E$ is contained in $\overline{v}\overline{K}$.        
 
\pars The paper is structured as follows: in Section \ref{Section min pair key pols} we give a quick tour of the theories of minimal pairs of definition and key polynomials. The definitions of pseudo monotone sequences and their limits are provided in Section \ref{Section pms}. In Section \ref{Section valns induced by pms} we define valuations induced by pseudo monotone sequences and study some of their basic properties. Section \ref{Section extns of v_E} deals with the extensions of $v_E$ to $\overline{K}(X)$. The definition of pure extensions is provided in Section \ref{Section pure extn}, where we also observe that $(K(X)|K,v_E)$ is a pure extension whenever $E$ is not a pseudo convergent sequence without a limit in $K$. The remaining case is considered in Section \ref{Section pcs of alg type}. Question \ref{question} is also investigated in this section, under the assumption that $E$ is pseudo convergent of algebraic type. The extension $(K(X)|K,v_E)$ when $E$ is pseudo divergent is studied in greater detail in Section \ref{Section E is a pds}. Finally, we define the valuation theoretic concept of rank and compute the rank of the valuation $v_E$ in Section \ref{Section rank v_E}.


\section*{Acknowledgements}
This work was supported by the Post-Doctoral Fellowship of the National Board of Higher Mathematics, India. The author would like to thank Giulio Peruginelli for directing his attention towards possible interrelation between implicit constant fields and dominating degrees, and for pointing out a mistake in an earlier version of this paper.


\section*{Notations}
A valued field $(K,v)$ is a field $K$ endowed with a valuation $v$. The value group of $K$ is denoted by $vK$ and the residue field by $Kv$. The value of an element $a\in K$ is denoted by $va$. By an extension of valued fields $(L|K,v)$ we will mean that $L|K$ is an extension of fields, $v$ is a valuation on $L$ and $K$ is equipped with the restricted valuation. If $L$ and $K$ are subfields of a larger valued field $(F,v)$, then we will also mean that both $L$ and $K$ are equipped with the restricted valuations.


\section{Minimal pairs and key polynomials}\label{Section min pair key pols}
A valued field extension $(K(X)|K,v)$ satisfies the famous Abhyankar inequality:
\[ \rr \frac{vK(X)}{vK} + \trdeg[K(X)v:Kv]\leq 1, \]
where $\rr vK(X)/vK$ is the $\QQ$-dimension of the divisible hull $\QQ\tensor_{\ZZ}(vK(X)/vK)$. This inequality is a consequence of [\ref{Bourbaki}, Chapter VI, \S 10.3, Theorem 1]. The extension $(K(X)|K,v)$ is said to be \textbf{valuation transcendental} if equality holds in the above inequality. The extension is said to be \textbf{value transcendental} if $\rr vK(X)/vK = 1$ and is said to be \textbf{residue transcendental} if $\trdeg [K(X)v:Kv] = 1$.  

\pars Take a polynomial $f(X)\in K[X]$ and $a\in K$. Write $f(X) = \sum_{i=0}^{n} c_i (X-a)^i$ where $c_i\in K$. Take an element $\a$ in some ordered abelian group extending $vK$. We define the map $v_{a,\a} : K[X] \longrightarrow vK + \ZZ\a$ by setting $v_{a,\a} f:= \min\{ v c_i + i\a \}$ and extend it canonically to $K(X)$. Then $v_{a,\a}$ is a valuation transcendental extension of $v$ from $K$ to $K(X)$ [\ref{Kuh value groups residue fields rational fn fields}, Lemma 3.10]. Further, $v_{a,\a}$ is value transcendental if and only if $\a\notin\QQ\tensor_{\ZZ}vK$ [\ref{Kuh value groups residue fields rational fn fields}, Lemma 3.10]. The valuation $v_{a,\a}$ is said to be a \textbf{monomial valuation}. It is known that the extension $(K(X)|K,v)$ is value (residue) transcendental if and only if $(\overline{K}(X)|\overline{K},\overline{v})$ is also value (residue) transcendental [\ref{Kuh value groups residue fields rational fn fields}, Lemma 3.3], where $\overline{v}$ is some extension of $v$ to $\overline{K}(X)$. Further, $(\overline{K}(X)|\overline{K},\overline{v})$ is valuation transcendental if and only if the valuation $\overline{v}$ on $\overline{K}(X)$ is a monomial valuation [\ref{Kuh value groups residue fields rational fn fields}, Theorem 3.11]. If we have $\overline{v}=\overline{v}_{a,\a}$ on $\overline{K}(X)$ for some $a\in\overline{K}$ and $\a\in \overline{v}\overline{K}(X)$, then we say that $(a,\a)$ is a \textbf{pair of definition for $v$ over $K$}. It can be easily checked that $(a,\a)$ is a pair of definition for $v$ over $K$ if and only if
\[ \a = \overline{v}(X-a) = \max \overline{v}(X-\overline{K}), \]  
where $\overline{v}(X-\overline{K}) := \{ \overline{v}(X-c)\mid c\in\overline{K} \}$ (see [\ref{Dutta min fields implicit const fields}, Section 2] for a proof). A valuation transcendental extension can have multiple pairs of definition. The following result, proved in [\ref{AP sur une classe}, Proposition 3], gives us a relation between the different pairs of definition:

\begin{Lemma}
	Take $a,a^\prime$ in $\overline{K}$ and $\a,\a^\prime$ in some ordered abelian group containing $\overline{v}\overline{K}$. Then,
	\[ \overline{v}_{a,\a} = \overline{v}_{a^\prime,\a^\prime} \text{ if and only if } \a=\a^\prime \text{ and } \overline{v}(a-a^\prime)\geq \a. \]
\end{Lemma}

\begin{Definition}
	A pair $(a,\a)\in\overline{K}\times \overline{v}\overline{K}(X)$ is said to be a \textbf{minimal pair of definition for $v$ over $K$} if the following conditions are satisifed:
	\sn (MP1) $\a = \overline{v}(X-a) = \max \overline{v}(X-\overline{K})$,
	\n (MP2) $\overline{v}(a-b)\geq\a\Longrightarrow [K(b):K]\geq [K(a):K]$ for all $b\in\overline{K}$.
\end{Definition}

Take a polynomial $f(X)\in K[X]$. We define,
\[ \d(f):=\max\{ \overline{v}(X-c)\mid c\text{ is a root of }f\}.  \]
It has been observed in [\ref{Novacoski key poly and min pairs}] that $\d(f)$ is independent of the choice of the extension of $v$ to $\overline{K}(X)$. A monic polynomial $Q(X)\in K[X]$ is said to be a \textbf{key polynomial for $v$ over $K$} if $\d(f)<\d(Q)$ for all $f(X)\in K[X]$ such that $\deg f < \deg Q$. Observe that a monic linear polynomial is a key polynomial by definition. 

\pars Given a polynomial $f(X)\in K[X]$ and a monic polynomial $Q(X)\in K[X]$, we have a unique expansion $f = \sum_{i=0}^{n} f_i Q^i$, where $f_i (X)\in K[X]$ with $\deg f_i < \deg Q$. We define a map $v_Q: K[X]\longrightarrow vK(X)$ by setting $v_Q(f):= \min \{ v f_i + ivQ \}$, and extend it canonically to $K(X)$. It is known that $v_Q$ is a valuation on $K(X)$ whenever $Q$ is a key polynomial for $v$ over $K$ [\ref{Nova Spiva key pol pseudo convergent}, Proposition 2.6]. A family $\L$ of key polynomials for $v$ over $K$ is said to form a \textbf{complete sequence of key polynomials for $v$ over $K$} if the following conditions are satisfied:
\sn (CSKP1) $\d(Q)\neq \d(Q^\prime)$ for $Q,Q^\prime\in\L$ with $Q\neq Q^\prime$,
\n (CSKP2) $\L$ is well ordered with respect to the ordering given by $Q<Q^\prime$ if and only if $\d(Q)<\d(Q^\prime)$,
\n (CSKP3) for any $f(X)\in K[X]$, there exists some $Q\in\L$ such that $vf = v_Qf$. 


\begin{Lemma}\label{Lemma v(y-K)}
	Let $(K(y)|K,v)$ be an extension of valued fields. Assume that $\a\in v(y-K)$ such that $\a\notin vK$. Then $\a = \max v(y-K)$.
\end{Lemma}

\begin{proof}
	Take $a\in K$ such that $\a:= v(y-a)\notin vK$. Suppose that there exists $b\in K$ such that $v(y-b) > v(y-a) = \a$. It then follows from the triangle inequality that $\a = v(a-b)\in vK$, which contradicts our assumption. 
\end{proof}


\section{Pseudo monotone sequences}\label{Section pms}
For the remainder of this article, we will assume that $(K,v)$ is a valued field and $E:=\{z_{\nu}\}_{\nu<\l}$ is a sequence in $K$ where $\l$ is some limit ordinal. 

\begin{Definition}
	The sequence $E$ is said to be a \textbf{pseudo convergent sequence} (pcs) if 
	\[ v(z_{\nu_1} - z_{\nu_2}) < v(z_{\nu_2} - z_{\nu_3}) \text{ for all } \nu_1 < \nu_2 < \nu_3 < \l. \]
\end{Definition}

Observe that $v(z_\nu - z_\mu) = v(z_\nu - z_{\nu+1})$ for all $\mu>\nu$. We set $\d_\nu:= v(z_\nu - z_{\nu+1})$. By definition, $\{\d_\nu\}_{\nu<\l}$ is a strictly increasing sequence. 

\begin{Definition}
	The sequence $E$ is said to be \textbf{pseudo divergent sequence} (pds) if 
	\[ v(z_{\nu_1} - z_{\nu_2}) > v(z_{\nu_2} - z_{\nu_3}) \text{ for all } \nu_1 < \nu_2 < \nu_3 < \l. \]
\end{Definition}

Observe that $v(z_\nu - z_\mu) = v(z_\nu - z_{\mu^\prime})$ for all $\mu^\prime, \mu<\nu$. We set $\d_\nu:= v(z_\nu - z_\mu) \text{ for all } \mu<\nu$. By definition, $\{\d_\nu\}_{\nu<\l}$ is a strictly decreasing sequence.

\begin{Definition}
	 The sequence $E$ is said to be a \textbf{pseudo constant sequence} (pcts) if
	\[ v(z_\nu - z_\mu) = v(z_{\nu^\prime} - z_{\mu^\prime}) \text{ for all } \nu\neq \mu, \, \nu^\prime\neq\mu^\prime. \]
\end{Definition}

We define $\d:= v(z_\nu - z_\mu) \text{ for all } \nu\neq\mu$. If we are making general statements about the sequence $E$ without differentiating between the cases, we will write $\d_\nu$ to mean $\d$ in the pcts case.

\begin{Definition}
	A pcs $E$ is said to be a \textbf{Cauchy sequence} if $\{ \d_\nu \}_{\nu<\l}$ is cofinal in $vK$. A pds $E$ will be said to be \textbf{diverging to infinity} if $\{\d_\nu \}_{\nu<\l}$ is coinitial in $vK$.
\end{Definition}

\begin{Definition}
	The sequence $E$ is said to be \textbf{pseudo monotone sequence} (pms) if is either a pcs, or pds, or pcts. 
\end{Definition}

\begin{Definition}
	An element $y\in K$ is said to be a \textbf{limit} of $E$ if 
	\[ v(y-z_\nu) = \d_\nu \text{ for all } \nu\text{ sufficiently large}. \]
\end{Definition}
 
Hence $y$ is a limit if the sequence  
\[ \{ v(y-z_\nu) \}_{\nu<\l} \text{ is } \begin{cases}
	\text{ultimately strictly increasing,} &\quad\text{when } E \text{ is a pcs,}\\
	\text{ultimately strictly decreasing,} &\quad\text{when } E \text{ is a pds,}\\
	\text{ultimately constant,} &\quad\text{when } E \text{ is a pcts.}
\end{cases} \]
Observe that a limit of $E$ is not necessarily uniquely determined. Here are some straightforward observations. 

\begin{Remark}
	Assume that $y$ is a limit of $E$. If $E$ is a pcs, then $\{ v(y-z_\nu) \}_{\nu<\l}$ is strictly increasing. If $E$ is a pcts, then $v(y-z_\nu) = \d$ for all but at most element of $E$. Both of these are consequences of the triangle inequality. Further, every element of $E$ is a limit of $E$ when $E$ is a pds or pcts. When $E$ is a pcs, then no element of $E$ is a limit of $E$.
\end{Remark}

\begin{Lemma}\label{Lemma limit}
	Let $(L|K,v)$ be an extension of valued fields and $E$ a pms in $(K,v)$. Assume that $y$ is a limit of $E$. Take any $b\in L$. Then $b$ is also a limit of $E$ if and only if $v(y-b)\geq\d_\nu$ for all $\nu$ sufficiently large.
\end{Lemma}

\begin{proof}
	The fact that $y$ is a limit of $E$ implies that $v(y-z_\nu) = \d_\nu$ for all $\nu$ sufficiently large. By definition, $b$ is also a limit of $E$ if and only if $v(b-z_\nu) = \d_\nu = v(y-z_\nu)$ for all $\nu$ sufficiently large. The result now follows from the ultrametric inequality. 
\end{proof}

\begin{Lemma}\label{Lemma limit or const seq}
	Let $E$ be a pms in $(K,v)$ and take $y\in K$. Then either $y$ is a limit of $E$ or the sequence $S:= \{ v(y-z_\nu) \}_{\mu<\l}$ is ultimately constant.
\end{Lemma}

\begin{proof}
	If $E$ is a pcts, then the sequence $S$ is ultimately constant either way.
	
	\pars We now assume that $E$ is a pds and $y$ is not a limit. Suppose that $S$ is not ultimately constant. If $S$ is ultimately increasing, then we can find $\nu_1 < \nu_2 < \nu_3$ such that
	\[ v(y-z_{\nu_1}) <v(y-z_{\nu_2}) < v(y-z_{\nu_3}).  \]
	By the triangle inequality we obtain
	\[ \d_{\nu_2} = v(y-z_{\nu_1}) < v(y-z_{\nu_2}) = \d_{\nu_3} \]
	which contradicts the fact that the sequence $\{ \d_\nu\}_{\nu<\l}$ is strictly decreasing. If $S$ is ultimately decreasing, then the assumption that $y$ is not a limit implies that $S$ is not ultimately strictly decreasing. So we obtain $\nu_1 < \nu_2<\nu_3$ such that 
	\[ v(y-z_{\nu_1}) > v(y-z_{\nu_2}) = v(y-z_{\nu_3}). \]
	It now follows from the triangle inequality that $\d_{\nu_2} = \d_{\nu_3}$ which is again a contradiction. Thus $S$ is neither ultimately increasing nor ultimately decreasing. Now take any $\nu<\l$. The fact that $S$ is not ultimately increasing implies that there exists some $\mu_1 > \nu$ such that $v(y-z_\nu) > v(y-z_{\mu_1})$. The fact that $S$ is not ultimately decreasing implies that there exists $\mu_2 >\mu_1$ such that $v(y-z_{\mu_1}) < v(y-z_{\mu_2})$. Hence,
	\[ v(y-z_{\nu}) > v(y-z_{\mu_1}) < v(y-z_{\mu_2}). \] 
	We now obtain from the triangle inequality that $\d_{\mu_1} = v(y-z_{\mu_1}) = \d_{\mu_2}$ which again yields a contradiction. Thus $S$ is ultimately constant.
	
	\pars We now assume that $E$ is a pcs and $y$ is not a limit. Suppose that $S$ is not ultimately constant. If $S$ is ultimately increasing, then the assumption that $y$ is not a limit implies that $S$ is not ultimately strictly increasing. So we have a relation 
	\[ v(y-z_{\nu_1}) = v(y-z_{\nu_2}) < v(y-z_{\nu_3}). \]
	But this yields $\d_{\nu_1} = \d_{\nu_2}$ which is a contradiction. If $S$ is ultimately decreasing, then we have a relation
	\[ v(y-z_{\nu_1}) > v(y-z_{\nu_2}) > v(y-z_{\nu_3}) \]
	which yields $\d_{\nu_1} = v(y-z_{\nu_2}) > v(y-z_{\nu_3}) = \d_{\nu_2}$. But this contradicts the fact that $\{ \d_\nu \}_{\nu<\l}$ is strictly increasing. Hence $S$ is neither ultimately increasing nor ultimately decreasing. Thus we have a relation 
	\[ v(y-z_{\nu}) > v(y-z_{\mu_1}) < v(y-z_{\mu_2}) \]
	where $\nu<\mu_1<\mu_2$. The triangle inequality again leads to the contradiction $\d_{\nu} = v(y-z_{\mu_1}) = \d_{\mu_1}$. Hence $S$ is ultimately constant.
\end{proof}

	Fix an extension $\overline{v}$ of $v$ to $\overline{K}$. Take a polynomial $f(X)\in K[X]$. Write
	\[ f(X) = c(X-a_1)\dotsc(X-a_j)\dotsc(X-a_n), \]
	where $c\in K$ and $a_1,\dotsc,a_j$ are all the \textbf{root-limits} of $(f,E)$, that is, they are all the roots of $f$ which are also limits of $E$. It follows from Lemma \ref{Lemma limit or const seq} that $\{\overline{v}(z_\nu-a_i)\}_{\nu<\l}$ is ultimately constant for all $i\geq j+1$. We denote the ultimate constant value of $\{\overline{v}(z_\nu-a_i)\}_{\nu<\l}$ by $\b_i$. Then,  
	\[ vf(z_\nu) = j\d_\nu + \b \text{ for all } \nu \text{ sufficiently large}, \]
	where $\b := vc + \b_{j+1} + \dotsc + \b_n$. The fact that $vf(z_\nu), \d_\nu \in v K$ implies that 
	\[ \b\in  vK. \]
	The above observation naturally extends to rational functions. Take $\phi(X)\in K(X)$. Then there exist $d\in\ZZ$ and $\b\in vK$ such that 
	\[ v\phi(z_\nu) = d\d_\nu + \b \text{ for all } \nu \text{ sufficiently large}. \]
	The number $d$ is referred to as the \textbf{dominating degree of $\phi$ with respect to $E$} in [\ref{Peruginelli, Spirito - extend valns pseudo monotone}] and is denoted by $\deg\dom_E(\phi)$. Observe that if $\phi = f/g$ where $f,g\in K[X]$, then 
	\[  \deg\dom_E (\phi) = d = \#\{ \text{root-limits of } (f,E) \} - \#\{ \text{root-limits of } (g,E) \}. \]
	Finally, observe that $v \phi(z_\nu) - v\phi(z_{\nu+1}) = d(\d_\nu-\d_{\nu+1})$ for all $\nu$ sufficiently large. Since $v \phi(z_\nu) - v\phi(z_{\nu+1})$ and $\d_\nu-\d_{\nu+1}$ are solely dependent on $v$, we conclude that 
	\[ \text{the dominating degree $d$ is independent of the choice of the extension of $v$ to }\overline{K}. \]

	If $E$ is a pcs or pds, then $\{\d_\nu\}_{\nu<\l}$ is either strictly increasing or strictly decreasing. The following result is now straightfoward.
	
	\begin{Proposition}\label{Proposition vQ ultimately constant no root-limits}
		Take $\phi(X)\in K(X)$. Then $\{v\phi(z_\nu)\}_{\nu<\l}$ is ultimately constant if and only if either $E$ is a pcts, or $E$ is a pcs or pds such that $\deg\dom_E(\phi) = 0$.
	\end{Proposition}

It follows from our preceding discussions that when $E$ is a pcs, the sequence $\{ vf(z_\nu) \}_{\nu<\l}$ is either ultimately constant or ultimately strictly increasing, where $f(X)\in K[X]$. If the sequence is ultimately constant for all polynomials $f$ over $K$, then we say that $E$ is a \textbf{pcs of transcendental type}. In this case $E$ does not admit any algebraic limits. Otherwise, $E$ is said to be a \textbf{pcs of algebraic type}. In this case, we can find a polynomial $Q(X)\in K[X]$ of minimal degree such that $\{vQ(z_\nu)\}_{\nu<\l}$ is ultimately strictly increasing. Without any loss of generality, we can choose $Q$ to be monic. It can be observed that $Q(X)$ is irreducible over $K$. The polynomial $Q(X)$ is said to be an \textbf{associated minimal polynomial for $E$}.


\section{Valuations induced by pseudo monotone sequences}\label{Section valns induced by pms}
Let $E:= \{ z_\nu \}_{\nu<\l}$ be a pms in $(K,v)$. Define,
\begin{align*}
	V_E &:= \{ \phi(X)\in K(X) \mid v\phi(z_\nu) \geq 0 \text{ for all } \nu \text{ sufficiently large} \},\\
	M_E &:= \{ \phi(X)\in K(X) \mid v\phi(z_\nu) > 0 \text{ for all } \nu \text{ sufficiently large} \}. 
\end{align*}
Then $(V_E,M_E)$ is a valuation domain of $K(X)$ [\ref{Peruginelli, Spirito - extend valns pseudo monotone}, Theorem 3.4]. We denote the corresponding valuation by $v_E$. It follows from the definition that $(V_E,M_E)$ dominates $(V,M)$, where $(V,M)$ is the valuation domain of the valuation $v$. Hence $v_E$ is an extension of $v$ to $K(X)$.

\begin{Proposition}\label{Prop v_E(f) in vK iff vf(z) constant}
	Take $\phi(X)\in K(X)$. Then $\{v\phi(z_\nu)\}_{\nu<\l}$ is ultimately constant if and only if $v_E(\phi)\in vK$. Further, in this case $v_E(\phi)=\b$ where $v\phi(z_\nu)=\b$ for all $\nu$ sufficiently large.
\end{Proposition}

\begin{proof}
	 We first assume that $\{v\phi(z_\nu)\}_{\nu<\l}$ is ultimately constant and write the constant value as $\b$. Take $a\in K$ such that $va=\b$. Write $\phi^\prime(X):= \frac{1}{a}\phi(X)$. Then $v\phi^\prime(z_\nu) = v(1/a)+ v\phi(z_\nu)$. Thus $\{ v\phi^\prime(z_\nu) \}_{\nu<\l}$ is also ultimately constant attaining the value $\b - va = 0$ ultimately. So without any loss of generality we can assume that $\b = 0$, that is, $v \phi(z_\nu) = 0$ for all $\nu$ sufficiently large. It now follows from the definition that $\phi\in V_E\setminus M_E$, that is, $v_E(\phi)=0$.
	
	\pars We now assume that $v_E(\phi)\in vK$. Take $a\in K$ such that $va = v_E(\phi)$. Write $\phi^\prime(X):= \frac{1}{a}\phi(X)$. Then $v_E(\phi^\prime) = 0$, that is, $\phi^\prime\in V_E\setminus M_E$. It follows from the definition that $v(\phi^\prime(z_\nu)) =0$ for all $\nu$ sufficiently large. Consequently, $v\phi(z_\nu) = va$ for all $\nu$ sufficiently large and hence $\{v\phi(z_\nu)\}_{\nu<\l}$ is ultimately constant.
\end{proof}

Recall that $v \phi(z_\nu) = d\d_\nu +\b$ for all $\nu$ sufficiently large, where $\b\in vK$ and $d=\deg\dom_E(\phi)$. If $E$ is a pcts, then $\d_\nu=\d$ for all $\nu$ and hence $v_E(\phi)=d\d+\b\in vK$. If $E$ is a pcs or pds then $\{\d_\nu\}_{\nu<\l}$ is either strictly increasing or strictly decreasing. In this case $\{v\phi(z_\nu)\}_{\nu<\l}$ is ultimately constant if and only if $d=0$. It follows that 
\[ v_E(\phi)\in vK \text{ if and only if } E \text{ is a pcts, or $E$ is a pcs or pds and } \deg\dom_E(\phi)=0. \]
Conversely, take a polynomial $f(X)\in K[X]$. Then,
\[ v_E(f)\notin vK \text{ if and only if $E$ is a pcs or pds and } (f,E) \text{ admits root-limits}. \]
Now assume that $v_E(f)\notin vK$. Suppose that $v_E(f)$ is a torsion element modulo $vK$, that is, $nv_E(f)\in vK$ for some $n\in\NN$. Then $v_E(f^n)\in vK$ and consequently, no root of $f^n$ is a limit of $E$. It follows that no root of $f$ is a limit of $E$ which contradicts our assumption. Thus, 
\[ v_E(f)\notin vK \text{ if and only if } v_E(f) \notin \QQ\tensor_{\ZZ}vK. \]


\begin{Remark}\textbf{($E$ is a pcs)}
	Take any $\nu<\l$ and define $\phi(X):= \frac{X-z_{\nu+1}}{X-z_\nu}$. Observe that $v\phi(z_\mu) = \d_{\nu+1}-\d_\nu>0$ for all $\mu>\nu+1$. By definition, $v_E(\phi)>0$ and hence $v_E(X-z_{\nu+1}) > v_E(X-z_\nu)$. Thus the sequence $\{v_E(X-z_\nu)\}_{\nu<\l}$ is strictly increasing and hence $X$ is a limit of $E$.
	\pars If $E$ is a pcs of transcendental type, then it follows from [\ref{Kaplansky}, Theorem 2] that $(K(X)|K,v_E)$ is an immediate extension.
\end{Remark}

\begin{Remark}\textbf{($E$ is a pds)}
	Recall that every element of $E$ is a limit of $E$. It follows from the preceding discussions that $v_E (X-z_\nu)\notin \QQ\tensor_{\ZZ}vK$. Further, $v_E(X-z_\nu) \in v_E(X-K)$. It follows from Lemma \ref{Lemma v(y-K)} that
	\[ v_E(X-z_\nu) = \max v_E(X-K) \notin \QQ\tensor_{\ZZ}vK \text{ for all }\nu<\l. \]
	As a consequence we obtain that $\{v_E(X-z_\nu)\}_{\nu<\l}$ is a constant sequence. In particular, $X$ is not a limit of $E$. Suppose that $v_E (X-z_\mu) \geq \d_\nu$ for some $\mu,\nu<\l$. The fact that $\{\d_\nu\}_{\nu<\l}$ is strictly decreasing implies that $v_E (X-z_\mu) > \d_\nu$ for all $\nu$ sufficiently large. In light of Lemma \ref{Lemma limit}, the observation that $z_\mu$ is a limit of $E$ then implies that $X$ is a limit of $E$, contradicting our observations above. It follows that 
	\[ v_E (X-z_\nu) < \d_\nu \text{ for all }\nu<\l. \]
\end{Remark}

\begin{Remark}\textbf{($E$ is a pcts)}
	Observe that $v(z_\mu-z_\nu) = \d$ for all $\mu\neq \nu$. It follows from Proposition \ref{Prop v_E(f) in vK iff vf(z) constant} that $v_E (X-z_\nu) = \d$ for all $\nu<\l$. As a consequence, we obtain that $X$ is a limit of $E$. 
\end{Remark}


\section{Extending $v_E$ to $\overline{K}(X)$}\label{Section extns of v_E}

Fix an extension $\overline{v}$ of $v$ to $\overline{K}$. Observe that $E$ is a pms in $(\overline{K},\overline{v})$. Consider the valuation $\overline{v}_E$ extending $\overline{v}$ to $\overline{K}(X)$ with the corresponding valuation domain given as follows:
\begin{align*}
	\overline{V}_E  &:= \{ \phi(X)\in\overline{K}(X)\mid \overline{v}\phi(z_\nu)\geq 0 \text{ for all }\nu \text{ sufficiently large} \},\\
	\overline{M}_E  &:= \{ \phi(X)\in\overline{K}(X)\mid \overline{v}\phi(z_\nu)> 0 \text{ for all }\nu \text{ sufficiently large} \}.
\end{align*}
We can directly observe that $(\overline{V}_E, \overline{M}_E)$ dominates $(V_E,M_E)$. Hence $\overline{v}_E$ is an extension of $v_E$ to $\overline{K}(X)$ [\ref{ZS2}, Ch VI, \S6, Lemma 1].

\begin{Proposition}\label{Prop v_E is monomial valn}
	Let $E$ be a pms in $(K,v)$ and assume that $a\in\overline{K}$ is a limit of $E$. Set $\a:= \overline{v}_E(X-a)$. Then 
	\[ \overline{v}_E = \overline{v}_{a,\a}. \]
\end{Proposition}

\begin{proof}
	We first assume that $E$ is a pcs or pds. The assumption that $a$ is a limit of $E$ implies that $\overline{v}_E(X-a)\notin \overline{v}\overline{K}$. Consider a polynomial $f(X)\in\overline{K}[X]$ and write $f(X) = \sum_{i=0}^{n} c_i (X-a)^i$ where $c_i\in\overline{K}$. We observe that all the monomials have distinct values. It follows that $\overline{v}_E f = \min\{ \overline{v}c_i +i\a \} $ and hence $\overline{v}_E = \overline{v}_{a,\a}$.
	
	\pars We now assume that $E$ is a pcts. Observe that the fact that $E$ is a pcts implies that $\a\in\overline{v}\overline{K}$. It follows from Proposition \ref{Prop v_E(f) in vK iff vf(z) constant} that $\overline{v}_E(X-a) = \overline{v}(z_\nu-a)$ for all $\nu$ sufficiently large. The assumption that $a$ is a limit of $E$ then implies that 
	\[ \a = \overline{v}_E(X-a) =\d. \]
	Take any $b\in\overline{K}$. Then $\overline{v}_E(X-b) = \overline{v}_E(X-a+a-b) \geq \min\{ \d, \overline{v}(a-b) \} = \overline{v}_{a,\d}(X-b)$. If $\d\neq\overline{v}(a-b)$, then it follows from the triangle inequality that $\overline{v}_E(X-b) = \overline{v}_{a,\d}(X-b)$. Now assume that $\overline{v}(a-b) =\d$. It follows from Lemma \ref{Lemma limit} that $b$ is also a limit of $E$. By definition, $\overline{v}(b-z_\nu) = \d$ for all $\nu$ sufficiently large. Consequently, it follows from Proposition \ref{Prop v_E(f) in vK iff vf(z) constant} that $\overline{v}_E(X-b) = \d$. Thus,
	\[ \overline{v}_E(X-b) = \overline{v}_{a,\d}(X-b) \text{ for all }b\in\overline{K}. \]
	Since any polynomial over $\overline{K}$ splits into linear factors over $\overline{K}$, it follows that $\overline{v}_E = \overline{v}_{a,\d} = \overline{v}_{a,\a}$. 
\end{proof}

\begin{Corollary}
	Assume that $E$ is a pms in $(K,v)$ with a limit $a\in K$. Set $\a:= v_E(X-a)$. Then $v_E = v_{a,\a}$.
\end{Corollary}

\begin{Remark}\label{Remark v(X-a) = v(X-b)}
	Assume that $E$ is a pcs or pds in $(K,v)$ and $a\neq b$ are elements in $\overline{K}$ which are limits of $E$. It follows that $\a:= \overline{v}_E (X-a)\notin \overline{v}\overline{K}$ and $\overline{v}_E (X-b)\notin \overline{v}\overline{K}$. In light of Lemma \ref{Lemma v(y-K)} we conclude that $\a =  \overline{v}_E (X-a) = \max \overline{v}_E (X-\overline{K}) = \overline{v}_E (X-b)$. Consequently, $\overline{v}(a-b)>\a$.	
\end{Remark}

We will write $\a>\overline{v}\overline{K}$ if $\a>\b$ for all $\b\in\overline{v}\overline{K}$. Analogously we will write $\a<\overline{v}\overline{K}$ to mean that $\a<\b$ for all $\b\in\overline{v}\overline{K}$.

\begin{Proposition}\label{Prop unique pair of defn}
	Let $E$ be a pms in $(K,v)$ and assume that $a\in\overline{K}$ is a limit of $E$. Set $\a:= \overline{v}_E(X-a)$. Then
	\[ \a>\overline{v}\overline{K} \text{ if and only if } E \text{ is a Cauchy sequence of algebraic type}. \]
	In this case $a$ is purely inseparable over $K^h$. On the other hand,
	\[  \a<\overline{v}\overline{K} \text{ if and only if } E \text{ diverges to infinity}.  \]
\end{Proposition}

\begin{proof}
	We first assume that $\a>\overline{v}\overline{K}$. The fact that $\a\notin\overline{v}\overline{K}$ implies that $E$ is a pcs or pds. Suppose that there exists some $b\in\overline{K}\setminus\{a\}$ such that $b$ is also a limit of $E$. It follows from Remark \ref{Remark v(X-a) = v(X-b)} that $\overline{v}(a-b)>\a$. However, this contradicts the assumption that $\a>\overline{v}\overline{K}$. So $a$ is the unique algebraic limit of $E$. As a consequence, $E$ is a pcs of algebraic type. Now suppose that $\{\d_\nu\}_{\nu<\l}$ is not cofinal in $\overline{v}\overline{K}$. Then there exists some $c\in\overline{K}$ such that $\overline{v}(a-c)>\d_\nu$ for all $\nu<\l$. But this implies that $c$ is a limit of $E$ which again yields a contradiction. Thus $\{\d_\nu\}_{\nu<\l}$ is cofinal in $\overline{v}\overline{K}$. The fact that $vK$ is cofinal in $\overline{v}\overline{K}$ then implies that $\{\d_\nu\}_{\nu<\l}$ is cofinal in $vK$. Hence $E$ is a Cauchy sequence in $(K,v)$ of algebraic type.
	
	\pars We now assume that $E$ is a Cauchy sequence of algebraic type, that is, $\{\d_\nu\}_{\nu<\l}$ is cofinal in $\overline{v}\overline{K}$. Recall that $X$ is a limit of $E$. The fact that $a$ is also a limit of $E$ then implies that $\a:= \overline{v}_E(X-a) > \d_\nu$ for all $\nu<\l$. It follows that $\a>\overline{v}\overline{K}$. 
	
	\pars We now assume that $E$ is a Cauchy sequence of algebraic type with the unique algebraic limit $a\in\overline{K}$. Suppose that $a$ is not purely inseparable over $K^h$. Then there exists some conjugate $\s a\neq a$ of $a$ over $K^h$ where $\s\in\Gal(\overline{K}|K^h)$. The fact that $(K^h,\overline{v})$ is henselian implies that $\overline{v}\circ\s = \overline{v}$ on $\overline{K}$. It follows that
	\[ \overline{v}(\s a-z_\nu) = (\overline{v}\circ\s) (a-z_\nu) = \overline{v}(a-z_\nu) = \d_\nu \text{ for all }\nu<\l. \]
	Consequently, $\s a$ is also a limit of $E$, which contradicts the uniqueness of $a$. Hence $a$ is purely inseparable over $K^h$.
	
	\parm We now assume that $\a<\overline{v}\overline{K}$. Hence either $E$ is a pcs or $E$ is a pds. Take any $c\in\overline{K}$. The triangle inequality implies that $\overline{v}_E(X-c) = \a\notin\overline{v}\overline{K}$ and hence $c$ is a limit of $E$. In particular, every element of $E$ is a limit of $E$. Hence $E$ is a pds. Further, we take any $\b\in\overline{v}\overline{K}$ and $b\in\overline{K}$ such that $\overline{v}(a-b) = \b$. The observation that $b$ is a limit of $E$ implies that $\b\geq \d_\nu$ for all $\nu$ sufficiently large. Since this condition holds for all $\b\in\overline{v}\overline{K}$ we observe that $\{\d_\nu\}_{\nu<\l}$ is coinitial in $\overline{v}\overline{K}$, that is, $E$ diverges to infinity.
	
	\pars Conversely, we assume that $E$ is a pds such that $\{\d_\nu\}_{\nu<\l}$ is coinitial in $\overline{v}\overline{K}$. Suppose that $\a\geq \d_\nu$ for some and hence all sufficiently large $\nu$. It then follows from Lemma \ref{Lemma limit} that $X$ is a limit of $E$ which yields a contradiction. Hence $\a<\d_\nu$ for all $\nu<\l$, that is, $\a<\overline{v}\overline{K}$.
\end{proof}


\section{Pure extensions}\label{Section pure extn}
An extension $(K(X)|K,v)$ is said to be a \textbf{pure extension} if it satisfies one of the following conditions:
\sn (PE1) $v(X-a)$ is not a torsion element modulo $vK$ for some $a\in K$,
\n (PE2) $v b(X-a)=0$ and $b(X-a)v$ is transcendental over $Kv$ for some $a,b\in K$,
\n (PE3) $X$ is the limit of some pcs $\{z_\nu\}_{\nu<\l}$ of transcendental type in $(K,v)$.

Pure extensions are very well understood in terms of minimal pairs, key polynomials and implicit constant fields. 

\begin{Proposition}
	Let $(\overline{K(X)}|K,v)$ be an extension of valued fields such that $(K(X)|K,v)$ is pure. Then the following statements hold true:
	\sn (1) $IC(K(X)|K,v) = K^h$, 
	\n (2) if the extension $(K(X)|K,v)$ satisfies (PE1) or (PE2), then $(a,v(X-a))$ is a minimal pair of definition for $v$ over $K$ and $\{X-a\}$ forms a complete sequence of key polynomials for $v$ over $K$, 
	\n (3) if the condition (PE3) is satisfied, then $(K(X)|K,v)$ is immediate and $\{X-z_\nu\}_{\nu<\l}$ forms a complete sequence of key polynomials for $v$ over $K$.  
\end{Proposition}

\begin{proof}
	The first assertion is proved in [\ref{Kuh value groups residue fields rational fn fields}, Lemma 3.7].
	
	\pars We now assume that $v(X-a)$ is not a torsion element modulo $vK$, that is, $v(X-a)\notin v\overline{K}$. Take a polynomial $f(X)\in K[X]$ and write $f(X) = \sum_{i=0}^{n} c_i (X-a)^i$ where $c_i \in \overline{K}$. Observe that each of the monomials in this expression have distinct values. It follows that $vf = \min\{ v c_i +iv (X-a) \} = v_{a, v(X-a)}f$, that is, $(a,v(X-a))$ is a pair of definition for $v$ over $K$. The fact that $a\in K$ implies that $(a,v(X-a))$ is a minimal pair of definition for $v$ over $K$. Further, the observation $vf = v_{a, v(X-a)}f$ for all $f(X)\in K[X]$ implies that $v|_{K(X)} = v_Q$, where $Q(X):= X-a$. By definition, $\{X-a\}$ forms a complete sequence of key polynomials for $v$ over $K$.
	
	\pars We now assume that $v b(X-a) =0$ and $b(X-a)v$ is transcendental over $Kv$ for some $a,b\in K$. It follows from [\ref{Kuh value groups residue fields rational fn fields}, Lemma 2.10] that $v|_{K(X)} = v_Q$ where $Q(X):= X-a$. Hence $\{X-a\}$ forms a complete sequence of key polynomials for $v$ over $K$. Observe that $\d(Q) = v(X-a)$, since $Q$ is a linear polynomial. It now follows from [\ref{Novacoski key poly and min pairs}, Theorem 1.1] that $(a,v(X-a))$ forms a minimal pair of definition for $v$ over $K$.
	
	\pars Finally, we assume that $(K(X)|K,v)$ satisfies the condition (PE3). In light of [\ref{Kaplansky}, Theorem 2], we obtain that the extension $(K(X)|K,v)$ is immediate. Set $Q_\nu:= X-z_\nu \in K[X]$. By definition, $\{\d(Q_\nu)\}_{\nu<\l}$ is a strictly increasing sequence. Further, for any $f(X)\in K[X]$, the sequence $\{vf(z_\nu)\}_{\nu<\l}$ is ultimately constant. It has been observed in the proof of [\ref{Kaplansky}, Theorem 2] that $vf = v_{Q_\nu}f$ for all $\nu$ sufficeintly large. As a consequence, we obtain that $\{X-z_\nu\}_{\nu<\l}$ forms a complete sequence of key polynomials for $v$ over $K$.
\end{proof}

\begin{Theorem}\label{Thm K(X)|K is pure}
	Assume that either $E$ is a pcs of transcendental type or $E$ has a limit $a\in K$. Fix an extension $\overline{v}$ of $v$ to $\overline{K}$ and consider the common extension $\overline{v}_E$ of $\overline{v}$ and $v_E$ to $\overline{K}(X)$. Fix an extension of $\overline{v}_E$ to $\overline{K(X)}$. Then $(K(X)|K,v_E)$ is a pure extension and $IC(K(X)|K,v_E) = K^h$.
\end{Theorem}

\begin{proof}
	If $E$ is a pcs of transcendental type, then the extension $(K(X)|K,v_E)$ is pure by definition. Otherwise, we assume that there exists some $a\in K$ such that $a$ is a limit of $E$. We have observed in Proposition \ref{Prop v_E is monomial valn} that $(a,\a)$ is a pair of definition for $v_E$ over $K$, where $\a:= v_E (X-a)$. The fact that $a\in K$ implies that $(a,\a)$ is a minimal pair of definition. If $E$ is a pcs or pds, then the fact that $a$ is a limit of $E$ implies that $\a\notin\overline{v}\overline{K}$. Consequently, $(K(X)|K,v_E)$ is pure by definition. If $E$ is a pcts, then $\a\in vK$. Take some $b\in K$ such that $vb(X-a) = 0$. The fact that $(a,\a)$ is a minimal pair of definition for $v_E$ over $K$ where $\a\in vK$ implies that $(K(X)|K,v_E)$ is residue transcendental. It now follows from [\ref{APZ characterization of residual trans extns}, Theorem 2.1] that $b(X-a)v$ is transcendental over $Kv$. We have thus proved the theorem.
\end{proof}

Recall that every element of $E$ is also a limit of $E$ whenever $E$ is a pds or pcts. Thus the remaining case to consider is when $E$ is a pcs of algebraic type without a limit in $(K,v)$.


\section{$E$ is a pcs of algebraic type}\label{Section pcs of alg type}

Throughout this section, we assume that $\overline{v}$ is a fixed extension of $v$ to $\overline{K}$. Take the common extension $\overline{v}_E$ of $v_E$ and $\overline{v}$ to $\overline{K}(X)$. We fix an extension of $\overline{v}_E$ to $\overline{K(X)}$. Our next result modifies [\ref{Dutta min fields implicit const fields}, Theorem 9.2] which illustrates that we can obtain stronger results in our setting.  

\begin{Theorem}\label{Thm pcs of alg type}
	Assume that $E$ is a pcs of algebraic type without a limit in $K$. Take an associated minimal polynomial $Q(X)\in K[X]$. Then the following statements hold true:
	\sn (1) there exists some root $a$ of $Q$ such that $(a,\overline{v}_E (X-a))$ forms a minimal pair of definition for $v_E$ over $K$,
	\n (2) $\{Q_\nu\}_{\nu<\l} \union \{Q\}$ forms a complete sequence of key polynomials for $v_E$ over $K$, where $Q_\nu:= X-z_\nu$ for all $\nu<\l$,
	\n (3) $IC(K(X)|K,v_E) = K^h$. 
\end{Theorem}

\begin{proof}
	The assumption that $ Q(X)$ is an associated minimal polynomial for $E$ implies that $\{vQ(z_\nu)\}_{\nu<\l}$ is ultimately strictly increasing. Hence there exists some root $a$ of $Q$ such that $a$ is a limit of $E$ [Proposition \ref{Proposition vQ ultimately constant no root-limits}]. It follows from Propostion $\ref{Prop v_E is monomial valn}$ that $(a,\a)$ is a pair of defninition for $v_E$ over $K$, where $\a:= \overline{v}_E (X-a)$. Take another pair of definition $(a^\prime,\a)$ for $v_E$ over $K$ and consider the minimal polynomial $f(X)$ of $a^\prime$ over $K$. It follows from [\ref{Dutta min fields implicit const fields}, Theorem 9.2] that $a^\prime$ is also a limit of $E$. Consequently, $\{vf(z_\nu)\}_{\nu<\l}$ is ultimately strictly increasing. The minimality of $Q$ then implies that $\deg f \geq \deg Q$, that is, $[K(a^\prime):K]\geq [K(a):K]$. Hence $(a,\a)$ is a minimal pair of definition for $v_E$ over $K$.
	
	\pars The observation $\a = \max \overline{v}_E (X-\overline{K})$ implies that $\d(Q) = \a$. It now follows from [\ref{Novacoski key poly and min pairs}, Theorem 1.1] that $v_E = v_Q$ and $Q(X)$ is a key polynomial for $v_E$ over $K$. Recall that $X$ is a limit of $E$. Thus $\d(Q_\nu) = v_E (X-z_\nu) = \d_\nu$ for all $\nu<\l$. Hence $\{\d(Q_\nu)\}_{\nu<\l}$ is a strictly increasing sequence and $\d(Q)> \d(Q_\nu)$ for all $\nu<\l$. Take any polynomial $g(X)\in K[X]$ with $\deg g < \deg Q$. The minimality of $Q$ implies that $\{vg(z_\nu)\}_{\nu<\l}$ is ultimately constant. It follows from Proposition \ref{Prop v_E(f) in vK iff vf(z) constant} that $v_E g = vg(z_\nu)$ for all $\nu$ sufficiently large. As a consequence, $vg = v_{Q_\nu}g$ for all $\nu$ sufficiently large. We have thus obtained that $\{Q_\nu\}_{\nu<\l}\union\{Q\}$ forms a complete sequence of key polynomials for $v_E$ over $K$.
	
	\pars We set $k:= K^h$. Observe that $k(X)^h = K(X)^h$. As a consequence, 
	\[ IC(K(X)|K,v_E) = \overline{K}\sect K(X)^h = \overline{k}\sect k(X)^h = IC(k(X)|k,\overline{v}_E). \]
	Now, $E$ is a pcs of algebraic type in $(k,\overline{v})$ and $(a,\a)$ is a pair of definition for $\overline{v}_E$ over $k$. Take a minimal pair of definition $(b,\a)$ for $\overline{v}_E$ over $k$ such that $b$ is a limit of $E$. It follows from [\ref{Dutta min fields implicit const fields}, Theorem 1.1] that 
	\[ k\subseteq IC(k(X)|k,\overline{v}_E)\subseteq k(b). \]
	If $b\in k$, then we have $IC(k(X)|k,\overline{v}_E) = k$. If $\a>\overline{v}\overline{K}$, then we obtain from Proposition \ref{Prop unique pair of defn} that $b$ is purely inseparable over $k$. Consequently, it follows from [\ref{Dutta min fields implicit const fields}, Propostion 7.1] that $IC(k(X)|k,\overline{v}_E) = k$. We now consider the complementary case, that is, $\overline{v}k$ is cofinal in $\overline{v}_E k(X)$ and $b\notin k$. In light of [\ref{Dutta min fields implicit const fields}, Proposition 3.6] we can assume that $b$ is separable over $k$. Take any conjugate $\s b\neq b$ of $b$ over $k$ where $\s\in\Gal(\overline{k}|k)$. The fact that $(k,\overline{v})$ is henselian implies that $\overline{v}\circ \s = \overline{v}$ on $\overline{k}$. Then,
    \[ \overline{v}(\s b -z_\nu) = (\overline{v}\circ\s)(b-z_\nu) = \overline{v}(b-z_\nu) = \d_\nu \text{ for all }\nu<\l.   \]
    It follows that any conjugate of $b$ is also a limit of $E$. Thus $\overline{v}_E (X-\s b) = \a$ for all conjuates $\s b$ of $b$ over $k$. Take the minimal polynomial $q(X)\in k[X]$ of $b$ over $k$. Then
    \[ \overline{v}_E (q) = n\a \text{ where } n:= \deg q = [k(b):k]. \]      
    It follows from [\ref{Dutta min fields implicit const fields}, Remark 3.3] that $\overline{v}_E k(X) = \overline{v}k(b)\dirsum n\ZZ\a$. Further, the extension $(k(b,X)|k(b),\overline{v}_E)$ being pure implies that $\overline{v}_E k(b,X) = \overline{v}k(b)\dirsum \ZZ\a$. As a consequence, we obtain that
    \[ (\overline{v}_E k(b,X):\overline{v}_E k(X)) = n = [k(b):k]. \]
    Now, observe that $k(b,X)^h = k(X)^h (b)$. It follows that 
	\[ (\overline{v}_E k(b,X):\overline{v}_E k(X)) = [k(b):k]\geq [k(b,X)^h: k(X)^h] \geq (\overline{v}_E k(b,X):\overline{v}_E k(X)).  \]
	As a consequence,
	\[ [k(b,X)^h: k(X)^h] = [k(b):k]. \]
	It follows that $k(b)$ and $k(X)^h$ are linearly disjoint over $k$. Hence 
	\[ k(b)\sect k(X)^h = k. \]
	The fact that $IC(k(X)|k,\overline{v}_E) \subseteq k(b)$ now implies that $IC(k(X)|k,\overline{v}_E) \subseteq k(b)\sect k(X)^h = k$. Further, $IC(k(X)|k,\overline{v}_E)$ is a separable-algebraic extension of $k$ by definition. It then follows that $IC(k(X)|k,\overline{v}_E) =k$. We have thus obtained that
	\[ IC(K(X)|K,v_E) = K^h. \]
\end{proof}


\begin{Theorem}\label{Thm pcs alg equivalences}
	Let $(K(X)|K,w)$ be a value transcendental extension and $E:= \{z_\nu\}_{\nu<\l}$ a pcs of algebraic type in $(K,v)$. Take associated minimal polynomials $Q(X)\in K[X]$ and $Q^h(X)\in K^h[X]$ of $E$ over $K$ and $K^h$ respectively. Assume that $X$ is a limit of $E$. Then
	\[  \deg \dom_E(Q^h) = \deg Q^h. \]
	Consider the following statements:
	\sn (i) $w = v_E$.
	\n (ii) There exists some common extension $\overline{w}$ of $\overline{v}$ and $w$ to $\overline{K}(X)$ satisfying the property that $a$ is a limit of $E$ if and only if $(a,\overline{w}(X-a))$ is a pair of definition for $w$ over $K$.
	\n (iii) Take a pair of definition $(a,\overline{w}(X-a))$ for $w$ over $K$, where $\overline{w}$ is some common extension of $w$ and $\overline{v}$ to $\overline{K}(X)$. Then for all $\b\in\overline{v}\overline{K}$, $\b> \overline{w}(X-a)$ if and only if $\b>\d_\nu$ for all $\nu<\l$.
	\n (iv) $(a, \overline{w}(X-a))$ is a minimal pair of definition for $w$ over $K$ for some root $a$ of $Q$, where $\overline{w}$ is some common extension of $w$ and $\overline{v}$ to $\overline{K}(X)$.
	\n (v) $(a, \overline{w}(X-a))$ is a pair of definition for $w$ over $K$ for some root $a$ of $Q$, where $\overline{w}$ is some common extension of $w$ and $\overline{v}$ to $\overline{K}(X)$.
	\n (vi) $w(Q)\notin\overline{v}\overline{K}$.
	\n (vii) $IC(K(X)|K,w) = K^h$, where we fix an extension of $w$ to $\overline{K(X)}$.\\
	Then we have the chains of relations:
	\begin{align*}
		&(i)\Longleftrightarrow(ii)\Longleftrightarrow (iii)\Longrightarrow(iv)\Longleftrightarrow(v)\Longleftrightarrow (vi).\\
		&(i)\Longleftrightarrow(ii)\Longleftrightarrow (iii)\Longrightarrow(vii).\\
	\end{align*}
\end{Theorem}

\begin{proof}
	If $E$ admits a limit $b^\prime\in K^h$, then $Q^h(X):= X-b^\prime \in K^h[X]$ is a linear polynomial and hence we have $\deg \dom_E(Q^h) = 1 = \deg Q^h$. Otherwise, it follows from [\ref{Dutta min fields implicit const fields}, Theorem 9.2] that there exists some root $b$ of $Q^h$ which is a limit of $E$. By definition, $\overline{v}(b-z_\nu) = \d_\nu$ for all $\nu<\l$. The fact that $K^h$ is henselian implies that $\overline{v}\circ\s = \overline{v}$ for all $\s\in\Gal(\overline{K}|K^h)$. It follows that
	\[ \overline{v}(\s b-z_\nu) = (\overline{v}\circ\s) (b-z_\nu) = \overline{v}(b-z_\nu) = \d_\nu. \]
	Consequently, every conjugate of $b$ over $K^h$ is also a limit of $E$, that is, $\deg\dom_E(Q^h) = \deg Q^h$.
	
	\pars We assume that $w=v_E$ and take the common extension $\overline{v}_E$ of $\overline{v}$ and $v_E$ to $\overline{K}(X)$. We have observed in Section \ref{Section valns induced by pms} that $a$ is a limit of $E$ if and only if $\overline{v}_E(X-a)\notin\overline{v}\overline{K}$. It now follows from Lemma \ref{Lemma v(y-K)} that $a$ is a limit of $E$ if and only if $\overline{v}_E (X-a) = \max \overline{v}_E(X-\overline{K})$, that is, $(a,\overline{v}_E (X-a))$ is a pair of definition for $v_E$ over $K$. We thus have the assertion $(i)\Longrightarrow (ii)$.
	
	\pars We now assume that there exists some common extension $\overline{w}$ of $\overline{v}$ and $w$ to $\overline{K}(X)$ such that $a$ is a limit of $E$ if and only if $(a,\overline{w}(X-a))$ is a pair of definition for $w$ over $K$. Take a limit $a$ of $E$ and set $\a:= \overline{w}(X-a)$. The fact that $w$ is a value transcendental extension implies that $\a\notin\overline{v}\overline{K}$. Further, the fact that $X$ and $a$ are limits of $E$ implies that $\a\geq\d_\nu$ for all $\nu$ sufficiently large [Lemma \ref{Lemma limit}]. Since $E$ is a pcs, we have that $\{\d_\nu\}_{\nu<\l}$ is strictly increasing. Hence
	\[ \a>\d_\nu \text{ for all }\nu<\l. \]
	Take any $\b\in\overline{v}\overline{K}$. If $\d_\nu>\b$ for some and hence all $\nu$ sufficiently large, then it follows that $\a>\b$. Conversely, suppose that $\a>\b$ but $\d_\nu<\b$ for all $\nu<\l$. Take some $b\in\overline{K}$ such that $\overline{v}(a-b) = \b$. In light of Lemma \ref{Lemma limit}, the observation that $a$ is a limit of $E$ implies that $b$ is also a limit of $E$. However, the relation $\overline{v}(a-b)<\a$ implies that $\overline{w}(X-b) = \b< \a$ and hence $(b,\overline{w}(X-b))$ is not a pair of definition for $w$ over $K$, thereby contradicting our staring assumption. We have thus shown that for any $\b\in\overline{v}\overline{K}$,
	\[ \a>\b \text{ if and only if } \d_\nu>\b \text{ for all }\nu \text{ sufficiently large}. \]  
	As a consequence, we obtain that for fixed $d\in\ZZ$ and $\b\in\overline{v}\overline{K}$, we have that
	\[ d\a+\b\geq 0 \text{ if and only if } d\d_\nu+\b\geq 0 \text{ for all }\nu \text{ sufficiently large}. \]
	Take $c\in\overline{K}$ such that $c$ is not a limit of $E$. It follows from Lemma \ref{Lemma limit} that $\overline{w}(X-c)<\d_\nu$ for all $\nu$ sufficiently large. We then obtain from the triangle inequality that 
	\[ \overline{w}(X-c) = \overline{v}(z_\nu-c) \text{ for all }\nu \text{ sufficiently large}.\]
	We can then expand this observation to arbitrary polynomials over $K$. Take $f(X)\in K[X]$ and write $f(X) = \prod_{i=1}^{n} (X-a_i)$, where $a_1,\dotsc,a_j$ are all the root-limits of $(f,E)$. From the starting assumptions we observe that $\overline{w}(X-a_i) = \a$ for all $i=1,\dotsc,j$. Hence
	\[ wf = j\a+(\b_{j+1}+\dotsc \b_n), \]
	where $\b_r:= \overline{v}(z_\nu - a_r)$ for all $\nu$ sufficiently large, for all $r = j+1,\dotsc,n$. Observe that $j$ and $\b_{j+1}+\dotsc \b_n$ are uniquely determined by $f(X)$. The same arguments imply that for any rational function $\phi(X)\in K(X)$, we have uniquely determined $d\in \ZZ$ and $\b\in\overline{v}\overline{K}$ such that
	\[ w\phi = d\a +\b. \]
	Further, observe that
	\[ v\phi(z_\nu) = d\d_\nu +\b \text{ for all }\nu \text{ sufficiently large}. \]
	From our observations earlier it now follows that
	\[ w\phi\geq 0  \text{ if and only if } v\phi(z_\nu)\geq 0 \text{ for all }\nu \text{ sufficiently large}.  \]
	We have thus shown that $W=V_E$, where $W$ is the valuation domain corresponding to the valuation $w$. Hence we have proved the assertion $(ii)\Longrightarrow (i)$.
	
	\pars We now take a pair of definition $(a,\overline{w}(X-a))$ for $w$ over $K$ where $\overline{w}$ is a common extension of $w$ and $\overline{v}$ to $\overline{K}(X)$, and set $\a := \overline{w}(X-a)$. The fact that $X$ is a limit of $E$ implies that $w(X-z_\nu) = \d_\nu$ for all $\nu<\l$. It now follows from Lemma \ref{Lemma v(y-K)} that $\a>\d_\nu$ for all $\nu<\l$, and as a consequence $a$ is also a limit of $E$. Take any $\b\in\overline{v}\overline{K}$ and $b\in\overline{K}$ such that $\overline{v}(a-b)=\b$. In light of Lemma \ref{Lemma limit} we observe that $\b>\d_\nu$ for all $\nu<\l$ if and only if $b$ is also a limit of $E$. On the other hand, $\b>\a$ if and only if $(b,\overline{w}(X-b))$ is a pair of definition for $w$ over $K$. We thus have the equivalence $(iii)\Longleftrightarrow(ii)$.

	\pars If there exists some $a\in K$ which is a limit of $E$, then $Q(X) = X-a$ and $(a,v_E(X-a))$ is a minimal pair of definition for $v_E$ over $K$. If $E$ does not have a limit in $K$ then the assertion $(i)\Longrightarrow (iv)$ is proved in Theorem \ref{Thm pcs of alg type}, where we consider the common extension $\overline{v}_E$ of $\overline{v}$ and $v_E$ to $\overline{K}(X)$. The assertion $(iv)\Longrightarrow (v)$ follows by definition. The assertion $(v)\Longrightarrow(vi)$ is straightforward. Conversely, we assume that $w(Q)\notin \overline{v}\overline{K}$. Thus $\a:= \overline{w}(X-a)\notin \overline{v}\overline{K}$ for some root $a$ of $Q$ where $\overline{w}$ is a common extension of $w$ and $\overline{v}$ to $\overline{K}(X)$. It follows from Lemma \ref{Lemma v(y-K)} that $\a= \max \overline{w}(X-\overline{K})$ and hence $(a,\a)$ is a pair of definition for $w$ over $K$. Observe that $\d_\nu = w(X-z_\nu)\in w(X-K)$. As a consequence, $\a>\d_\nu \text{ for all }\nu<\l$. Recall that $X$ is a limit of $E$. It now follows from Lemma \ref{Lemma limit} that $a$ is a limit of $E$. We take another pair of definition $(b,\a)$ for $w$ over $K$. Then $\overline{v}(a-b) > \a>\d_\nu$ for all $\nu<\l$ and hence $b$ is also a limit of $E$. Consequently, the sequence $\{ vf(z_\nu) \}_{\nu<\l}$ is ultimately strictly increasing, where $f(X)\in K[X]$ is the minimal polynomial of $b$ over $K$ [Proposition \ref{Proposition vQ ultimately constant no root-limits}]. The minimality of $Q(X)$ then implies that $\deg f \geq \deg Q$, that is, $[K(b):K]\geq [K(a):K]$. By definition, $(a,\a)$ is a minimal pair of definition for $w$ over $K$, thereby proving the assertion $(vi)\Longrightarrow (iv)$. We have thus obtained the equivalence $(iv)\Longleftrightarrow (v)\Longleftrightarrow (vi)$.
	
	\pars The assertion $(i)\Longrightarrow (vii)$ follows from Theorem \ref{Thm pcs of alg type}.
\end{proof}

The following result is now immediate:

\begin{Corollary}
	Let $(K(X)|K,w)$ be a value transcendental extension and $E$ a Cauchy sequence of algebraic type in $(K,v)$. Assume that $X$ is a limit of $E$. Then $w=v_E$.
\end{Corollary}

The assertions $(iv)\Longrightarrow(i)$ and $(vii)\Longrightarrow (i)$ in the statement of Theorem \ref{Thm pcs alg equivalences} are not always true. An example illustrating this fact is furnished below.

\begin{Example}\label{Example 3.6 not 1}
	Take a prime number $p>0$. Consider the valued field $(\FF_p(t),v_t)$ where $v:= v_t$ is the $t$-adic valuation. Fix an extension of $v$ to $\overline{\FF_p(t)}$ which we again denote by $v$. We define a sequence of elements $\{a_i\}_{i\in\NN}\subset \overline{\FF_p(t)}$ recursively as follows:
	\[ a_0 := -\frac{1}{t}, \,\,  a_{i+1}^p - a_{i+1} + a_i = 0. \]
	Set $K:= \FF_p(t)^h(a_i\mid i\in\NN)$ and take $\eta\in\overline{K}$ such that $\eta^p = \frac{1}{t}$. Then $(K,v)$ is a henselian valued field with $vK = \frac{1}{p^\infty}\ZZ := \{ \frac{a}{p^n} \mid a\in\ZZ, n\in\NN \}$. It has been observed in [\ref{Kuh defect}, Example 3.14] that $(K(\eta)|K,v)$ is an immediate extension and 
	\[  v(\eta-K) = \{ \g\in vK \mid \g<0 \}. \]
	It follows from [\ref{Kaplansky}, Theorem 1] that there exists a pcs $E:= \{z_\nu\}_{\nu<\l}$ in $(K,v)$ without a limit in $K$ such that $\eta$ is a limit of $E$. Hence $\d_\nu = v(\eta-z_\nu) < 0$ for all $\nu<\l$, where $\d_\nu:= v(z_\nu - z_{\nu+1})$. The fact that $\eta$ is algebraic over $K$ implies that $E$ is a pcs of algebraic type. Further, it follows from Proposition \ref{Proposition vQ ultimately constant no root-limits} that $\{ vQ(z_\nu)_{\nu<\l} \}$ is ultimately strictly increasing, where $Q(X):= X^p - \frac{1}{t} \in K[X] $ is the minimal polynomial of $\eta$ over $K$. The facts that $\deg Q$ is a prime number and $E$ does not have a limit in $K$ imply that
	\[ Q(X) \text{ is an associated minimal polynomial of $E$ over }K. \]
	Take an irrational number $\a>0$ and consider the value extension $\overline{w}:= v_{\eta,\a}$ of $v$ from $\overline{K}$ to $\overline{K}(X)$. Set $w:= \overline{w}|_{K(X)}$. In light of Lemma \ref{Lemma limit} we observe that
	\[ X \text{ is a limit of }E.\]
	Now, $w(Q) = w(X-\eta)^p = p\a\notin v\overline{K}$. From Theorem \ref{Thm pcs alg equivalences} we now obtain that
	\[ (\eta,\a) \text{ is a minimal pair of definition for $w$ over }K. \] 
	As a consequence, we obtain from [\ref{Dutta min fields implicit const fields}, Proposition 7.1] that
	\[ IC(K(X)|K,w) = K. \]
	We obtain from [\ref{Dutta min fields implicit const fields}, Remark 3.3] that $wK(X) = vK(\eta)\dirsum\ZZ w(Q)\subset \RR$. Consequently, $\rk wK(X) = 1$.
	Observe that $\{\d_\nu\}_{\nu<\l}$ is cofinal in $v(\eta-K) = \{ \g\in vK \mid \g<0  \}$. Further, given any negative real number $r\in\RR$, there exists $n\in\NN$ such that $r < -\frac{1}{p^n} < 0$. It follows that $v(\eta-K)$ and hence consequently $\{\d_\nu\}_{\nu<\l}$ is cofinal in $0^- := \{ r\in\RR\mid r<0\}$. Thus $\sup\{\d_{\nu}\}_{\nu<\l} = 0$. It now follows from Theorem \ref{Thm rank v_E} that $\rk v_E K(X) = \rk vK +1 = 2$. As a consequence, 
	\[ w\neq v_E. \]
\end{Example}

We now provide a couple of examples to illustrate that the conditions $(iv)$ and $(vii)$ in the setting of Theorem \ref{Thm pcs alg equivalences} may not be interrelated. 

\begin{Example}
	Take the valued field $(K,v)$, $\eta\in\overline{K}$ and a pcs $E$ in $(K,v)$ with limit $\eta$ as in Example \ref{Example 3.6 not 1}. Recall that $v(\eta-K) = \{ \g\in vK\mid \g<0 \}$. Further, $Q(X):= X^p - \frac{1}{t}$ is an associated minimal polynomial of $E$ over $K$. Take $a\in\overline{K}$ such that
	\[ a^{p^2} = \frac{1}{t^p} +t. \]
	Then $K(a)|K$ is a purely inseparable extension and $[K(a):K]\leq p^2$. Observe that $a^p = \frac{1}{t} + \frac{1}{\eta}$. The fact that $\eta\notin K$ implies that $a^p\notin K$. It follows that
	\[ [K(a):K] = p^2. \] 
	Applying Frobenius endomorphism, we obtain that $(a-\eta)^{p^2} = t$. As a consequence, 
	\[ v(a-\eta) = \frac{1}{p^2}. \]
	It now follows from Lemma \ref{Lemma limit} that $a$ is a limit of $E$. Take the value transcendental extension $\overline{w}:= v_{a,\a}$ of $v$ from $\overline{K}$ to $\overline{K}(X)$, where $\a>\frac{1}{p^2}$ is an irrational number. Set $w:= \overline{w}|_{K(X)}$. Then $(a,\a)$ is a pair of definition for $w$ over $K$. Applying Lemma \ref{Lemma limit} again, we obtain that $X$ is also a limit of $E$. Further, it follows from [\ref{Dutta min fields implicit const fields}, Lemma 5.1] that $IC(K(X)|K,w)\subseteq K(a)$. The observation that $IC(K(X)|K,w)$ is a separable-algebraic extension of $K$ while $K(a)|K$ is purely inseparable implies that
	\[ IC(K(X)|K,w) = K. \]
	From the triangle inequality we obtain that $\overline{w}(X-\eta) = v(a-\eta) = \frac{1}{p^2} < \a$. Consequently, $(\eta,\overline{w}(X-\eta))$ is not a pair of definition for $w$ over $K$. 
\end{Example}

\begin{Example}
	Let $p>0$ be a prime number. Consider the valued field $(\FF_p(t),v_t)$ where $v:= v_t$ is the $t$-adic valuation. Fix an extension of $v$ to $\overline{\FF_p(t)}$ which we again denote by $v$. Set $(K,v)$ to be the perfect hull of the henselization of $\FF_p(t)$. Take $a\in\overline{K}$ such that $a^p - a =\frac{1}{t}$. It has been observed in [\ref{Kuh defect}, Example 3.12] that $(K(a)|K,v)$ is an immediate extension with
	\[ v(a-K) = \{ \g\in vK \mid \g<0 \}. \]
	In light of [\ref{Kaplansky}, Thoerem 1] we can take a pcs $E$ in $(K,v)$ without any limits in $K$ such that $a$ is a limit of $E$. The fact that $a\in\overline{K}$ implies that $E$ is a pcs of algebraic type [\ref{Kaplansky}, Theorem 2]. Take the minimal polynomial $Q(X):= X^p - X - \frac{1}{t} \in K[X]$ of $a$ over $K$. It follows from Proposition \ref{Proposition vQ ultimately constant no root-limits} that $\{ vQ(z_\nu) \}_{\nu<\l}$ is ultimately strictly increasing. Since $\deg Q$ is a prime number, we conclude that
	\[ Q(X) \text{ is an associated minimal polynomial of $E$ over }K.  \]
	Take an irrational number $\a > 0$ and consider the value transcendental extension $\overline{w}:= v_{a,\a}$ of $v$ from $\overline{K}$ to $\overline{K}(X)$. Set $w:= \overline{w}|_{K(X)}$. Applying Lemma \ref{Lemma limit}, we obtain that $X$ is also a limit of $E$. Further, $(a,\a)$ is a pair of definition for $w$ over $K$. It then follows from Theorem \ref{Thm pcs alg equivalences} that $(a,\a)$ is a minimal pair of definition for $w$ over $K$. Observe that the set $\{ a+i \mid i\in \FF_p \}$ is the set of all the conjugates of $a$ over $K$. Consequently, $\kras(a,K) = 0< \a$, where $\kras(a,K)$ is the Krasner constant of $a$ over $K$. It now follows from [\ref{Dutta min fields implicit const fields}, Theorem 1.3] that
	\[ IC(K(X)|K,w) = K(a). \]
\end{Example}


\section{$E$ is a pds}\label{Section E is a pds}

Throughout this section we will assume that $E:= \{ z_\nu \}_{\nu<\l}$ is a pds and we fix an extension $\overline{v}$ of $v$ to $\overline{K}$. Recall that $\{\d_\nu\}_{\nu<\l}$ is a strictly decreasing sequence. 


\begin{Theorem}\label{Thm E pds}
	Let $(K(X)|K,w)$ be a value transcendental extension and $E$ a pds in $(K,v)$. Assume that $X$ is not a limit of $E$. Then the following are equivalent:
	\sn (i) $w = v_E$.
	\n (ii) There exists some common extension $\overline{w}$ of $w$ and $\overline{v}$ to $\overline{K}(X)$ satisfying the property that $(a,\overline{w}(X-a))$ is a pair of definition for $w$ over $K$ if and only if $a$ is a limit of $E$.
	\n (iii) $(z_\mu,w(X-z_\mu))$ is a pair of definition for $w$ over $K$ for some $z_\mu\in E$. Further, for all $\b\in\overline{v}\overline{K}$ we have that $\b< w (X-z_\mu)$ if and only if $\b<\d_\nu$ for all $\nu<\l$.	
\end{Theorem}

\begin{proof}
	The assertion $(i)\Longrightarrow (ii)$ follows from our discussions in Section \ref{Section valns induced by pms}.
	
	\pars We assume that $\overline{w}$ is a common extension such that $(a,\overline{w}(X-a))$ is a pair of definition for $w$ over $K$ if and only if $a$ is a limit of $E$. Take a limit $a$ of $E$ and set $\a:= \overline{w}(X-a)$. Then $\a\notin\overline{v}\overline{K}$. The fact that each element of $E$ is a limit of $E$ implies that $\overline{w}(X-z_\nu) =\a$ for each $\nu<\l$. Suppose that $\a>\d_\nu$ for some and hence all $\nu$ sufficiently large. It follows consequently from Lemma \ref{Lemma limit} that $X$ is a limit of $E$ which contradicts our starting assumption. Hence,
	\[ \a<\d_\nu \text{ for all }\nu<\l. \]
	Take any $\b\in\overline{v}\overline{K}$ and $b\in\overline{K}$ such that $\overline{v}(a-b) = \b$. Suppose that $\a<\b<\d_\nu$ for all $\nu$. It follows that $(b,\a)$ is a pair of definition for $w$ over $K$ but $b$ is not a limit of $E$, contradicting our starting assumptions. Hence, for all $\b\in\overline{v}\overline{K}$,
	\[ \a>\b \text{ if and only if } \d_\nu>\b \text{ for all }\nu<\l. \]
	As a consequence, for fixed $d\in\ZZ$ and $\b\in\overline{v}\overline{K}$ we obtain that,
	\[ d\a + \b\geq 0 \text{ if and only if }d\d_\nu+\b \geq 0 \text{ for all }\nu<\l. \]
	Take any $c\in\overline{K}$ which is not a limit of $E$. It follows that $\overline{w}(X-c) < \a$ and hence 
	\[ \overline{w}(X-c) = \overline{v}(z_\nu-c) \text{ for all }\nu<\l. \]
	The above observations imply that for any rational function $\phi(X)\in K(X)$, we have uniquely determined $d\in\ZZ$ and $\b\in\overline{v}\overline{K}$ such that 
	\begin{align*}
		w\phi &= d\a + \b,\\
		v \phi(z_\nu) &= d\d_\nu +\b \text{ for all } \nu\text{ sufficiently large}.
	\end{align*}  
It now follows that $v\phi(z_\nu)\geq 0$ for all $\nu$ sufficiently large whenever $w\phi\geq 0$. Conversely, assume that $v\phi(z_\nu)\geq 0$ for all $\nu$ sufficiently large. Then $\{\d_\nu\}_{\nu<\l}$ being a strictly decreasing sequence implies that $v\phi(z_\nu)\geq 0$ for all $\nu<\l$ and hence $w\phi\geq 0$. As a consequence, we obtain $W = V_E$ where $W$ is the valuation domain corresponding to the valuation $w$. We have thus proved the assertion $(ii)\Longrightarrow (i)$.

\pars We now assume that $(z_\mu, w(X-z_\mu))$ is a pair of definition for $w$ over $K$ for some $z_\mu\in E$. Set $\a:= w(X-z_\mu)$. Take a common extension $\overline{w}$ of $w$ and $\overline{v}$ to $\overline{K}(X)$. Take any $\b\in\overline{v}\overline{K}$ and $b\in\overline{K}$ such that $\overline{v}(z_\mu-b) = \b$. The fact that $\{\d_\nu\}_{\nu<\l}$ is strictly decreasing implies that $\b\geq\d_\nu$ for all $\nu$ sufficiently large if and only if $\b\geq \d_\nu$ for some $\nu<\l$. It now follows from Lemma \ref{Lemma limit} that $b$ is not a limit of $E$ if and only if $\b<\d_\nu$ for all $\nu<\l$. On the other hand, $\b<\a$ if and only if $(b,\overline{w}(X-b))$ is not a pair of definition for $w$ over $K$. We thus have the equivalence $(iii)\Longleftrightarrow(ii)$. 
\end{proof}

\begin{Corollary}
	Let $(K(X)|K,w)$ be a value transcendental extension and $E$ a pds in $(K,v)$ diverging to infinity. Assume that $X$ is not a limit of $E$. Then $w = v_E$.
\end{Corollary}

\begin{proof}
	Suppose that $w(X-z_\mu) \geq \d_\nu$ for some $z_\mu\in E$ and $\nu<\l$. The fact that $\{\d_\nu\}_{\nu<\l}$ is a decreasing sequence implies that $X$ is a limit of $E$, contradicting our assumption. It follows that $w(X-z_\mu) < \d_\nu$ for all $\nu<\l$. By definition, $\{\d_\nu\}_{\nu<\l}$ is coinitial in $\overline{v}\overline{K}$. Hence $w(X-z_\mu)\notin\overline{v}\overline{K}$ and $w(X-z_\mu)<\overline{v}\overline{K}$. It now follows from Theorem \ref{Thm E pds} that $w=v_E$.
\end{proof}


\section{Rank computation} \label{Section rank v_E}

 Let $\G$ be an ordered abelian group. A subset $\D$ of $\G$ is said to be an \textbf{isolated subgroup} if it satisfies the following conditions:
\begin{itemize}[noitemsep]
	\item $\D$ is a proper subgroup of $\G$,
	\item if $\a\in\D$, then $\b\in\D$ for all $\b\in\G$ with $-\a\leq\b\leq\a$.
\end{itemize}
Thus the set of all isolated subgroups of $\G$ are ordered by inclusion. The \textbf{rank of $\G$} is defined to be the ordinal type of the set of all isolated subgroups of $\G$ and we write it as $\rk \G$.

We now assume that $\G_1,\dotsc,\G_n$ are ordered abelian groups with $\rk\G_i = r_i<\infty$ for all $i=1,\dotsc,n$. Let the chain of isolated subgroups of $\G_i$ be given by
\[ (0) = H_{i,1}\subsetneq H_{i,2}\subsetneq \dotsc \subsetneq H_{i,r_i}. \]
Consider the ordered abelian group $\G:= (\G_1\dirsum\G_2\dirsum\dotsc\dirsum\G_n)_{\lex}$. It is a straightforward check that the chain of isolated subgroups of $\G$ is given by $H_{1,1}\dirsum \dotsc \dirsum H_{n,1} \subsetneq H_{1,1}\dirsum \dotsc \dirsum H_{n-1,1} \dirsum H_{n,2} \subsetneq \dotsc \subsetneq H_{1,1}\dirsum \dotsc \dirsum H_{n-1,1} \dirsum H_{n,r_n} \subsetneq H_{1,1}\dirsum \dotsc \dirsum H_{n-1,1} \dirsum \G_n \subsetneq H_{1,1}\dirsum \dotsc \dirsum H_{n-2,1} \dirsum H_{n-1, 2} \dirsum \G_n \subsetneq \dotsc \subsetneq H_{1,r_1}\dirsum \G_2\dirsum\dotsc\dirsum \G_n$. It follows that
\begin{equation}\label{eqn rank G dirsum H}
	\rk (\G_1\dirsum\G_2\dirsum\dotsc\dirsum\G_n)_{\lex} = \rk \G_n + \dotsc + \rk\G_1.
\end{equation}

\parm For the rest of the section, we fix an extension $\overline{v}$ of $v$ to $\overline{K}$. Given a pms $E$ in $(K,v)$ and assuming $\rk vK$ is a finite positive integer, our goal in this section is to compute $\rk v_E K(X)$. The extension $(K(X)|K,v_E)$ is immediate when $E$ is a pcs of transcendental type and is a pure residue transcendental extension when $E$ is a pcts [Theorem \ref{Thm K(X)|K is pure}]. Hence $v_E K(X) = vK$ in either case. So we will assume that $E$ is a pcs of algebraic type or $E$ is a pds. It is known that that the rank of a valuation is invariant under algebraic extensions [\ref{ZS2}, Chapter VI, \S 11, Lemma 2]. It follows that $\rk \overline{v}_E \overline{K}(X) = \rk v_E K(X)$ and $\rk\overline{v}\overline{K} = \rk vK$. 

\pars We first assume that $E$ is a pcs of algebraic type. Take $a\in\overline{K}$ such that 
\[ a \text{ is a limit of }E. \]
Let $\rk vK = n$ for some positive integer $n$. It has been observed in [\ref{Abh book}, Proposition 2.10] that $\overline{v}\overline{K}$ is order isomorphic to $(\G_1\dirsum \dotsc \dirsum \G_n)_{\lex}$, where $\G_i$ are subgroups of $\RR$. From now on, we will identify $\overline{v}\overline{K}$ with $(\G_1\dirsum \dotsc \dirsum \G_n)_{\lex}$. Assume that under this identification, 
\begin{equation}\label{eqn identification}
	\d_\nu \text{ corresponds to } (\d_{1,\nu}, \dotsc, \d_{n,\nu}) \text{ for all }\nu<\l.
\end{equation} 
Observe that the fact that $\overline{v}\overline{K}$ is a divisible group implies that $\G_i$ is a divisible group for each $i$. Hence, for any $r_i\in\RR$ we have that
\[ r_i\notin\G_i \Longleftrightarrow r_i \text{ is torsion free over }\G_i. \]
Recall that the Dedekind-MacNeille completion of $\QQ$ is $\tilde{\RR}:= \{-\infty\}\union\RR\union\{\infty\}$, where $\tilde{\RR}$ is endowed with an order extending the canonical order on $\RR$ such that $-\infty<r<\infty$ for all $r\in\RR$ [\ref{Schroder ordered sets}, Example 5.3.3]. By definition, $\tilde{\RR}$ is also the Dedekind-MacNeille completion of $\RR$. Hence any subset $A_i\subseteq \G_i$ has a unique supremum and infimum in $\tilde{\RR}$, which we will denote as $\sup A_i$ and $\inf A_i$ respectively. 

\pars The fact that $\{\d_\nu\}_{\nu<\l}$ is a strictly increasing sequence implies that $\{ \d_{1,\nu} \}_{\nu<\l}$ is an increasing sequence. Further, $E$ is a Cauchy sequence if and only if $\{\d_{1,\nu}\}_{\nu<\l}$ is cofinal in $\G_1$. Now the fact that $\G_1$ is a subgroup of the archimedean ordered group $\RR$ implies that $\G_1$ is cofinal in $\RR$. It follows that $\{\d_{1,\nu}\}_{\nu<\l}$ is cofinal in $\G_1$ if and only if $\{\d_{1,\nu}\}_{\nu<\l}$ is cofinal in $\RR$, that is, $\sup \{\d_{1,\nu} \}_{\nu<\l} = \infty$. Consider an embedding of $\overline{v}\overline{K}$ into the ordered abelian group $(\ZZ\dirsum \overline{v}\overline{K})_{\lex}$ given by $\b\mapsto (0,\b)$ for all $\b\in\overline{v}\overline{K}$. Set $\a:= (1,0)$ and take the extension $\overline{v}_{a,\a}$ of $\overline{v}$ to $\overline{K}(X)$. In light of Lemma \ref{Lemma limit} we observe that $X$ is a limit of $E$. It then follows from the corollary to Theorem \ref{Thm pcs alg equivalences} that $\overline{v}_{a,\a} = \overline{v}_E$. Observe further that $(\overline{K}(X)|\overline{K}, \overline{v}_E)$ is a pure value transcendental extension and hence $\overline{v}_E \overline{K}(X) = \overline{v}\overline{K} \dirsum\ZZ\a$. It follows that $\overline{v}_E\overline{K}(X) = (\ZZ\dirsum\overline{v}\overline{K})_{\lex}$. Hence, 
\[ \rk v_E K(X) = \rk vK +1 \text{ when $E$ is a Cauchy sequence}. \]    

\pars We now assume that $E$ is not a Cauchy sequence. Hence there exists some $r_1\in\RR$ such that
\[ r_1 = \sup\{\d_{1,\nu}\}_{\nu<\l}. \]
We first assume that $r_1\notin\G_1$. Hence $\d_{1,\nu} < r_1$ for all $\nu<\l$ and $\{\d_{1,\nu}\}_{\nu<\l}$ is not ultimately constant. Consider the ordered abelian group $(\G_1^\prime\dirsum\G_2\dirsum\dotsc\dirsum\G_n)_{\lex}$ where $\G_1^\prime := \G_1\dirsum\ZZ r_1$. Take $\a:= (r_1,0,\dotsc,0)\notin\overline{v}\overline{K}$ and the extension $\overline{v}_{a,\a}$ of $\overline{v}$ to $\overline{K}(X)$. Observe that $\a>\d_\nu$ for all $\nu<\l$. It thus follows from Lemma \ref{Lemma limit} that $X$ is a limit of $E$. Take any $(\b_1,\dotsc,\b_n)\in\overline{v}\overline{K}$. If $\b_1 = \d_{1,\nu}$ for some $\nu$, then the fact that $\{\d_{1,\nu}\}_{\nu<\l}$ is not ultimately constant implies that $\b_1 < \d_{1,\mu}$ for all $\mu$ sufficiently large. It follows that $(\b_1,\dotsc,\b_n) > \d_\nu$ for all $\nu<\l$ if and only if $\b_1 > \d_{1,\nu}$ for all $\nu<\l$. The facts that $r_1 = \sup\{\d_{1,\nu}\}_{\nu<\l}$, $r_1\notin \G_1$ and $\b_1\in\G_1$ imply that $\b_1 > \d_{1,\nu}$ for all $\nu<\l$ if and only if $\b_1 > r_1$. It follows that 
\[ (\b_1,\dotsc,\b_n) > \d_\nu \text{ for all }\nu<\l \text{ if and only if } (\b_1,\dotsc,\b_n) > \a. \]     
We can then conclude from Theorem \ref{Thm pcs alg equivalences} that $\overline{v}_{a,\a} = \overline{v}_E$. The observation $\overline{v}_E \overline{K}(X) = \overline{v}\overline{K} \dirsum\ZZ\a$ implies that $\overline{v}_E \overline{K}(X) = (\G_1^\prime\dirsum\G_2\dirsum\dotsc\dirsum\G_n)_{\lex}$. The fact that $\G_1^\prime$ is a subgroup of $\RR$ implies that $\rk \G_1^\prime = 1$. It then follows from (\ref{eqn rank G dirsum H}) that $\rk \overline{v}_E \overline{K}(X) = n$. As a consequence, 
\[ \rk v_E K(X) = \rk vK \text{ when } r_1 \in \RR\setminus \G_1. \]

\pars We now assume that $r_1 \in \G_1$. There are two cases: $r_1 > \d_{1,\nu}$ for all $\nu<\l$, and, $r_1 = \d_{1,\nu}$ for all $\nu\geq\nu_1$. We first assume that
\[ r_1 > \d_{1,\nu} \text{ for all }\nu<\l. \]
As a consequence, the sequence $\{\d_{1,\nu}\}_{\nu<\l}$ is not ultimately constant. Take any $(\b_1,\dotsc,\b_n)\in\overline{v}\overline{K}$. Our prior observations show that 
\[ (\b_1,\dotsc,\b_n) > \d_\nu \text{ for all } \nu<\l \text{ if and only if } \b_1 \geq r_1. \]
Consider an embedding of $\overline{v}\overline{K}$ into the ordered abelian group $(\G_1\dirsum\ZZ \dirsum \G_2\dirsum \dotsc \dirsum \G_n)_{\lex}$ given by $(\b_1,\dotsc, \b_n) \mapsto (\b_1, 0, \b_2, \dotsc , \b_n)$. Take $\a:= (r_1, -1, 0, \dotsc , 0)$. By definition, $(\b_1,0,\b_2,\dotsc,\b_n) > \a$ if and only if $\b_1\geq r_1$. We have thus obtained that 
\[ (\b_1,\dotsc,\b_n) > \d_\nu \text{ for all } \nu<\l \text{ if and only if } (\b_1,0,\b_2,\dotsc,\b_n) > \a. \]
Consider the valuation $\overline{v}_{a,\a}$ on $\overline{K}(X)$. The fact that $r_1 > \d_{1,\nu}$ for all $\nu<\l$ implies that $\a> \d_\nu$ for all $\nu<\l$. Consequently, $X$ is a limit of $E$. It then follows from Theorem \ref{Thm pcs alg equivalences} that $\overline{v}_{a,\a} = \overline{v}_E$. The observation $\overline{v}_E \overline{K}(X) = \overline{v}\overline{K} \dirsum\ZZ\a$ implies that $\overline{v}_E \overline{K}(X) = (\G_1\dirsum\ZZ\dirsum\G_2\dirsum\dotsc\dirsum\G_n)_{\lex}$. Hence, 
\[ \rk v_E K(X) = \rk vK +1 \text{ when } r_1\in\G_1\text{ and } r_1 > \d_{1,\nu} \text{ for all }\nu<\l. \]

\pars We now consider the final case, that is, $r_1 = \d_{1,\nu}$ for all $\nu\geq \nu_1$. The fact that $\{\d_\nu\}_{\nu<\l}$ is strictly increasing implies that $\{\d_{2,\nu}\}_{\nu\geq \nu_1}$ is an increasing sequence. We can now use the preceding arguments to draw similar conclusions. We briefly sketch the arguments in each case without going into the detailed checks. 

\pars We first assume that $\sup \{\d_{2,\nu}\}_{\nu\geq \nu_1} = \infty$. Consider the ordered abelian group $(\G_1\dirsum\ZZ\dirsum\G_2\dirsum\dotsc\dirsum\G_n)_{\lex}$ and take $\a:= (r_1,1,0,\dotsc,0)$. Embed $\overline{v}\overline{K}$ in $(\G_1\dirsum\ZZ\dirsum\G_2\dirsum\dotsc\G_n)_{\lex}$ by setting $(\b_1,\dotsc,\b_n)\mapsto (\b_1,0,\b_2,\dotsc,\b_n)$. Observe that $\a>\d_\nu$ for all $\nu<\l$. In this case we obtain that $\overline{v}_E \overline{K}(X) = (\G_1\dirsum\ZZ\dirsum\G_2\dirsum\dotsc\dirsum\G_n)_{\lex}$. Consequently, $ \rk v_E K(X) = \rk vK +1$. 

\pars We now assume that there exists $r_2\in\RR$ such that
\[ r_2 = \sup\{ \d_{2,\nu} \}_{\nu\geq \nu_1}. \]
Further, assume that $r_2\notin\G_2$. Take $\a:= (r_1,r_2,0,\dotsc,0)$ in the ordered abelian group $(\G_1\dirsum\G_2^\prime\dirsum\dotsc\dirsum\G_n)_{\lex}$, where $\G_2^\prime := \G_2 \dirsum \ZZ r_2$. The fact that $r_2\notin\G_2$ implies that $r_2>\d_{2,\nu}$ for all $\nu\geq\nu_1$ and consequently, $\a>\d_\nu$ for all $\nu<\l$. In this case we obtain that $\overline{v}_E \overline{K}(X) = (\G_1\dirsum\G_2^\prime\dirsum\dotsc\dirsum\G_n)_{\lex}$ and hence $\rk v_E K(X) = \rk vK$.

\pars We now assume that $r_2 \in \G_2$ and $r_2 > \d_{2,\nu}$ for all $\nu\geq \nu_1$. Hence $\{\d_{2,\nu}\}_{\nu\geq\nu_1}$ is not ultimately constant. Embed $\overline{v}\overline{K}$ in the ordered abelian group $(\G_1\dirsum\G_2\dirsum\ZZ\dirsum\G_3\dirsum\dotsc\dirsum\G_n)_{\lex}$ by setting $(\b_1,\dotsc,\b_n) \mapsto (\b_1,\b_2,0,\b_3,\dotsc,\b_n)$. Take $\a:= (r_1,r_2,-1,0,\dotsc,0)$. Observe that $\a>\d_\nu$ for all $\nu<\l$. Take any $(\b_1,\dotsc,\b_n)\in\overline{v}\overline{K}$. The fact that $\{\d_{2,\nu}\}_{\nu\geq\nu_1}$ is not a ultimately constant sequence implies that $(\b_1,\dotsc,\b_n) > \d_\nu$ for all $\nu<\l$ if and only if $\b_1 > r_1$, or, $\b_1 = r_1$ and $\b_2 > \d_{2,\nu}$ for all $\nu\geq \nu_1$. Observe that $\b_2 > \d_{2,\nu}$ for all $\nu\geq\nu_1$ if and only if $\b_2\geq r_2$. It follows that $(\b_1,\dotsc,\b_n) > \d_\nu$ for all $\nu<\l$ if and only if $(\b_1,\b_2) > (r_1,r_2)$. On the other hand, $(\b_1,\b_2,0,\b_3,\dotsc,\b_n) > \a = (r_1,r_2,-1,0,\dotsc,0)$ if and only if $(\b_1,\b_2) > (r_1,r_2)$. We have thus obtained that 
\[ (\b_1,\dotsc,\b_n) > \d_\nu \text{ for all } \nu<\l \text{ if and only if } (\b_1,\b_2,0,\b_3,\dotsc,\b_n) > \a. \]
Similar arguments as above now imply that $\overline{v}_E\overline{K}(X) = (\G_1\dirsum\G_2\dirsum\ZZ\dirsum\G_3\dirsum\dotsc\dirsum\G_n)_{\lex}$ and as a consequence, $\rk v_E K(X) = \rk vK+1$.

\pars The remaining case is when $r_2 = \d_{2,\nu}$ for all $\nu\geq\nu_2$. It follows that $\{\d_{3,\nu}\}_{\nu\geq\nu_2}$ is an increasing sequence. We can again study the separate cases in a similar fashion. We now assume that 
\begin{align*}
	\sup\{\d_{1,\nu}\}_{\nu<\l} &= r_1 = \d_{1,\nu} \text{ for all } \nu\geq \nu_1,\\
	\sup\{\d_{i,\nu}\}_{\nu\geq \nu_{i-1}} &= r_i = \d_{i,\nu} \text{ for all } \nu\geq \nu_i,
\end{align*}
for all $i=2,\dotsc,n-1$. Then the sequence $\{\d_{n,\nu}\}_{\nu\geq\nu_{n-1}}$ is an increasing sequence. The cases $\sup\{\d_{n,\nu}\}_{\nu\geq\nu_{n-1}} = \infty$ and $r_n := \sup\{\d_{n,\nu}\}_{\nu\geq\nu_{n-1}}  \notin \G_n$ are treated similarly as before. We now assume that $r_n \in \G_n$. If $r_n = \d_{n,\nu}$ for all $\nu\geq\nu_n$, then the sequnece $\{ (\d_{1,\nu},\dotsc,\d_{n,\nu})\}_{\nu<\l}$ is ultimately constant, contradicting the fact that $\{\d_\nu\}_{\nu<\l}$ is strictly increasing. Thus $r_n > \d_{n,\nu}$ for all $\nu\geq\nu_{n-1}$, in which case we obtain $\rk v_E K(X) = \rk vK +1$ as above. 
\newline We compile our observations in the tree diagram Figure \ref{fig E pcs base level}.

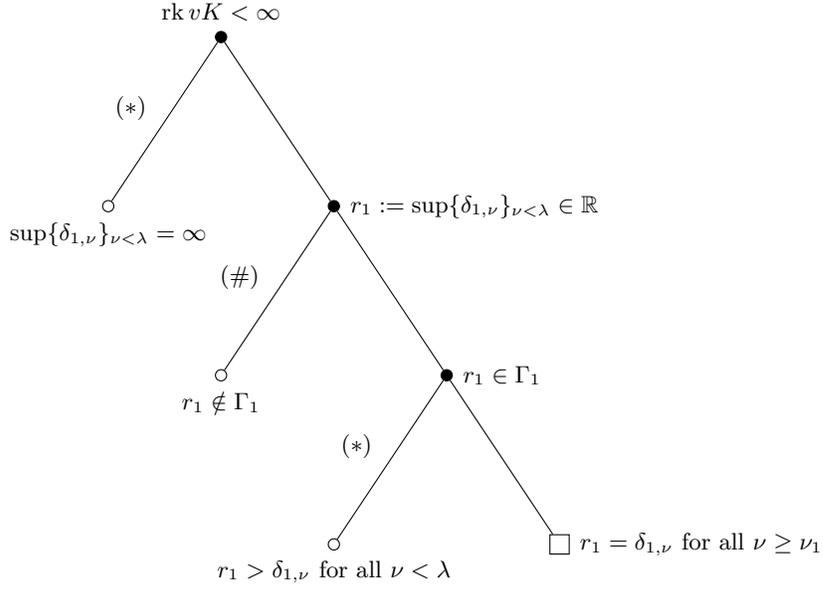
\begin{figure}[h]
	\begin{tikzpicture}[scale=1.5,font=\footnotesize]
		\tikzstyle{level 1}=[level distance=15mm, sibling distance=20mm]
		\tikzstyle{level 2}=[level distance=15mm, sibling distance=20mm]	
		\node(0) [solid node,label=above:{\textbf{$\rk vK <\infty$}}]{} 
		child { node [hollow node, label=below:{$\sup\{ \d_{1,\nu} \}_{\nu<\l} = \infty$}] {}
			edge from parent node [left,xshift=-3,yshift=5] {($*$)} }
		child { node [solid node,label=right:{$r_1:= \sup\{ \d_{1,\nu} \}_{\nu<\l} \in \RR$}] {}
			child { node [hollow node, label=below:{$r_1\notin\G_1$}] {}
				edge from parent node [left,xshift=-3,yshift=5] {($\#$)} }
			child { node [solid node,label=right:{$r_1\in\G_1$}] {}
				child { node [hollow node, label=below:{$r_1>\d_{1,\nu}$ for all $\nu<\l$}] {}
					edge from parent node [left,xshift=-3,yshift=5] {($*$)} }
				child { node [rectangle node, label=right:{$r_1=\d_{1,\nu}$ for all $\nu\geq\nu_1$}] {}
					edge from parent node [left,xshift=-3,yshift=5] {} }
				edge from parent node [left,xshift=-3,yshift=5] {} }
			edge from parent node [left,xshift=-3,yshift=5] {} };
	\end{tikzpicture}
\caption{$E$ pcs of alg type: base level}\label{fig E pcs base level}
\end{figure}

If we arrive at a hollow node in Figure \ref{fig E pcs base level} with a branch marked $(*)$ we have $\rk v_E K(X) = \rk vK +1$. If we arrive at a hollow node with a branch marked $(\#)$ we have $\rk v_E K(X) = \rk vK$. We keep going when we arrive at the square node. Assume that we have arrived at the $i$-th square node for some $i<n$. Note that this implies we have arrived at the square node in each of the previous steps. We then continue as in Figure \ref{fig E pcs i level}.   

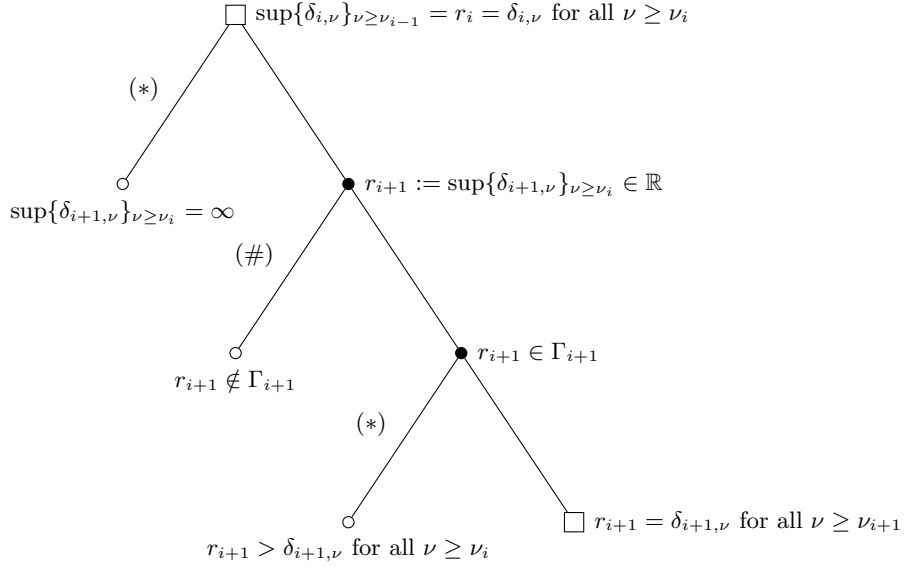
\begin{figure}[h]
	\centering
\begin{tikzpicture}[scale=1.5,font=\footnotesize]
	\tikzstyle{level 1}=[level distance=15mm, sibling distance=20mm]
	\tikzstyle{level 2}=[level distance=15mm, sibling distance=20mm]	
	\node(0) [rectangle node,label=right:{$\sup\{\d_{i,\nu}\}_{\nu\geq \nu_{i-1}} = r_i=\d_{i,\nu}$ for all $\nu\geq\nu_i$}]{} 
	child { node [hollow node, label=below:{$\sup\{ \d_{i+1,\nu} \}_{\nu\geq\nu_i} = \infty$}] {}
		edge from parent node [left,xshift=-3,yshift=5] {($*$)} }
	child { node [solid node,label=right:{$r_{i+1}:= \sup\{ \d_{i+1,\nu} \}_{\nu\geq\nu_i} \in \RR$}] {}
		child { node [hollow node, label=below:{$r_{i+1}\notin\G_{i+1}$}] {}
			edge from parent node [left,xshift=-3,yshift=5] {($\#$)} }
		child { node [solid node,label=right:{$r_{i+1}\in\G_{i+1}$}] {}
			child { node [hollow node, label=below:{$r_{i+1}>\d_{i+1,\nu}$ for all $\nu\geq\nu_i$}] {}
				edge from parent node [left,xshift=-3,yshift=5] {($*$)} }
			child { node [rectangle node, label=right:{$r_{i+1}=\d_{i+1,\nu}$ for all $\nu\geq\nu_{i+1}$}] {}
				edge from parent node [left,xshift=-3,yshift=5] {} }
			edge from parent node [left,xshift=-3,yshift=5] {} }
		edge from parent node [left,xshift=-3,yshift=5] {} };
\end{tikzpicture}
\caption{$E$ pcs of alg type: $(i+1)$-th level}\label{fig E pcs i level}
\end{figure}

\pars The dual result holds when $E$ is a pds in $(K,v)$. This is precisely due to the dual nature of Theorem \ref{Thm pcs alg equivalences}$(iii)$ and Theorem \ref{Thm E pds}$(iii)$. Since the arguments are completely similar to the case when $E$ is a pcs of algebraic type, we only state the final results here. Recall that we identify $\overline{v}\overline{K}$ with $(\G_1\dirsum\dotsc\dirsum\G_n)_{\lex}$ where $\G_i$ are subgroups of $\RR$. Further, for each $\nu<\l$, $\d_\nu$ corresponds to $(\d_{1,\nu}, \dotsc, \d_{n,\nu})$ under this identification. We then obtain the tree diagram Figure \ref{fig E pds base level}.

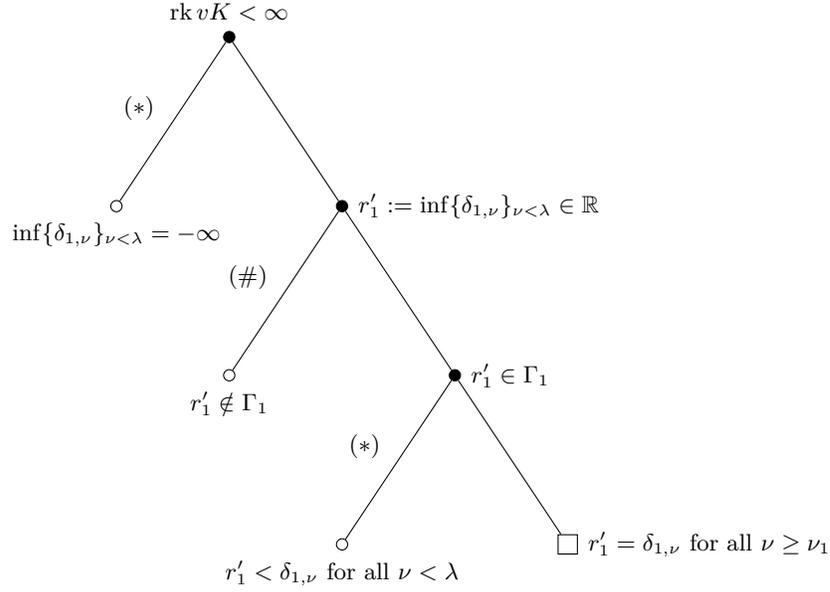
\begin{figure}[h]
	\centering
\begin{tikzpicture}[scale=1.5,font=\footnotesize]
	\tikzstyle{level 1}=[level distance=15mm, sibling distance=20mm]
	\tikzstyle{level 2}=[level distance=15mm, sibling distance=20mm]	
	\node(0) [solid node,label=above:{\textbf{$\rk vK <\infty$}}]{} 
	child { node [hollow node, label=below:{$\inf\{ \d_{1,\nu} \}_{\nu<\l} = -\infty$}] {}
		edge from parent node [left,xshift=-3,yshift=5] {($*$)} }
	child { node [solid node,label=right:{$r_1^\prime:= \inf\{ \d_{1,\nu} \}_{\nu<\l} \in \RR$}] {}
		child { node [hollow node, label=below:{$r_1^\prime\notin\G_1$}] {}
			edge from parent node [left,xshift=-3,yshift=5] {($\#$)} }
		child { node [solid node,label=right:{$r_1^\prime\in\G_1$}] {}
			child { node [hollow node, label=below:{$r_1^\prime<\d_{1,\nu}$ for all $\nu<\l$}] {}
				edge from parent node [left,xshift=-3,yshift=5] {($*$)} }
			child { node [rectangle node, label=right:{$r_1^\prime=\d_{1,\nu}$ for all $\nu\geq\nu_1$}] {}
				edge from parent node [left,xshift=-3,yshift=5] {} }
			edge from parent node [left,xshift=-3,yshift=5] {} }
		edge from parent node [left,xshift=-3,yshift=5] {} };
\end{tikzpicture}
\caption{$E$ pds: base level}\label{fig E pds base level}
\end{figure}

Whenever we arrive at a hollow node in Figure \ref{fig E pds base level} with a branch marked $(*)$ we have $\rk v_E K(X) = \rk vK+1$. We have $\rk v_E K(X) = \rk vK$ whenever we arrive at a hollow node with a brach marked $(\#)$. We arrive at the square node if we have arrived at the square node in each of the preceding steps. In this case we continue as in Figure \ref{fig E pds i level}.

\begin{figure}[h]
	\centering
\begin{tikzpicture}[scale=1.5,font=\footnotesize]
	\tikzstyle{level 1}=[level distance=15mm, sibling distance=20mm]
	\tikzstyle{level 2}=[level distance=15mm, sibling distance=20mm]	
	\node(0) [rectangle node,label=right:{$\inf\{\d_{i,\nu}\}_{\nu\geq \nu_{i-1}} = r_i^\prime=\d_{i,\nu}$ for all $\nu\geq\nu_i$}]{} 
	child { node [hollow node, label=below:{$\inf\{ \d_{i+1,\nu} \}_{\nu\geq\nu_i} = -\infty$}] {}
		edge from parent node [left,xshift=-3,yshift=5] {($*$)} }
	child { node [solid node,label=right:{$r_{i+1}^\prime:= \inf\{ \d_{i+1,\nu} \}_{\nu\geq\nu_i} \in \RR$}] {}
		child { node [hollow node, label=below:{$r_{i+1}^\prime\notin\G_{i+1}$}] {}
			edge from parent node [left,xshift=-3,yshift=5] {($\#$)} }
		child { node [solid node,label=right:{$r_{i+1}^\prime\in\G_{i+1}$}] {}
			child { node [hollow node, label=below:{$r_{i+1}^\prime<\d_{i+1,\nu}$ for all $\nu\geq\nu_i$}] {}
				edge from parent node [left,xshift=-3,yshift=5] {($*$)} }
			child { node [rectangle node, label=right:{$r_{i+1}^\prime=\d_{i+1,\nu}$ for all $\nu\geq\nu_{i+1}$}] {}
				edge from parent node [left,xshift=-3,yshift=5] {} }
			edge from parent node [left,xshift=-3,yshift=5] {} }
		edge from parent node [left,xshift=-3,yshift=5] {} };
\end{tikzpicture}
\caption{$E$ pds: $(i+1)$-th level}\label{fig E pds i level}
\end{figure}
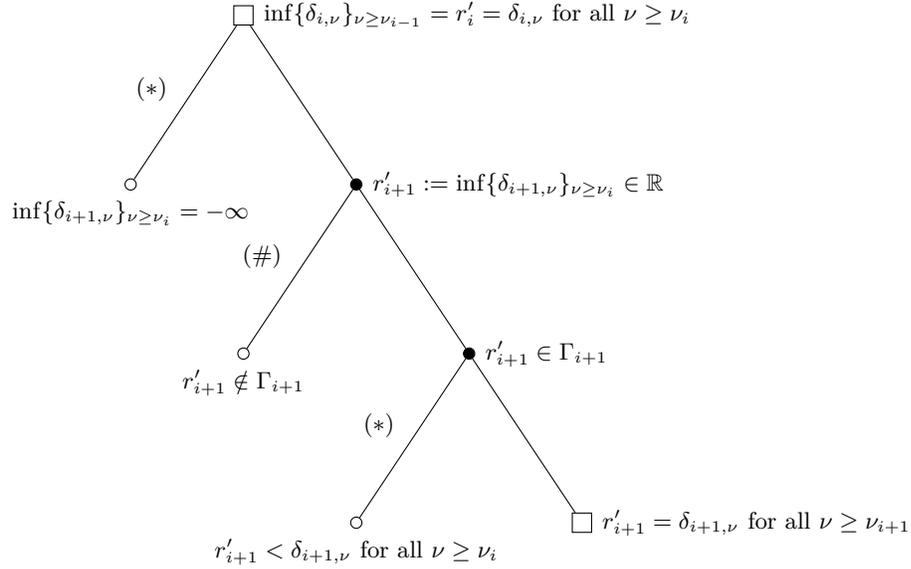

\pars Observe that $(\tilde{\RR}\dirsum\dotsc\dirsum\tilde{\RR})_{\lex}$ is a complete totally ordered set [\ref{Schroder ordered sets}, Proposition 9.16], where $\tilde{\RR}:= \{-\infty\}\union\RR\union\{\infty\}$. Hence the set $\mathcal{D}:= \{ (\d_{1,\nu}, \dotsc, \d_{n,\nu}) \}_{\nu<\l}$ has a unique supremum and infimum in $(\tilde{\RR}\dirsum\dotsc\dirsum\tilde{\RR})_{\lex}$. We define,
\[ \sup(E,\overline{v}):= \sup\mathcal{D} \text{ and } \inf(E,\overline{v}):= \inf\mathcal{D}. \]  

\pars We first assume that $E$ is a pcs of algebraic type. Hence $\mathcal{D}$ is strictly increasing. Let $r_1 := \sup\{\d_{1,\nu}\}_{\nu<\l}$ where $r_1\in\tilde{\RR}$ (we allow $r_1=\infty$ in this case). If $r_1>\d_{1,\nu}$ for all $\nu<\l$, then it is a straightforward check that $\sup (E,\overline{v}) = (r_1,-\infty,\dotsc,-\infty)$. If $r_1 = \d_{1,\nu}$ for all $\nu\geq\nu_1$, then we observe that $\{\d_{2,\nu}\}_{\nu\geq\nu_1}$ is an increasing sequence. Let $r_2:= \sup\{\d_{2,\nu}\}_{\nu\geq\nu_1}\in\tilde{\RR}$. If $r_2>\d_{2,\nu}$ for all $\nu\geq\nu_1$, then $\sup (E,\overline{v}) = (r_1,r_2,-\infty,\dotsc,-\infty)$. Otherwise, $r_2= \d_{2,\nu}$ for all $\nu\geq\nu_2$ and consequently $\{\d_{3,\nu}\}_{\nu\geq\nu_2}$ is an increasing sequence. Repeated iterations of this argument yield the following result:

\begin{Proposition}
	Assume that $E$ is a pcs of algebraic type in $(K,v)$ with $\rk vK=n$ and fix an extension $\overline{v}$ of $v$ to $\overline{K}$. Identify $\overline{v}\overline{K}$ with $(\G_1\dirsum\dotsc\dirsum\G_n)_{\lex}$ where $\G_i$ are subgroups of $\RR$. Consider the identification (\ref{eqn identification}). Then the following statements hold true:
	\sn (1) $\sup(E,\overline{v}) = (r_1,-\infty,\dotsc,-\infty)$ if and only if $\sup\{\d_{1,\nu}\}_{\nu<\l} = r_1 > \d_{1,\nu}$ for all $\nu<\l$.
	\n (2) $\sup(E,\overline{v}) = (r_1,\dotsc,r_i, -\infty,\dotsc,-\infty)$ for some $i\geq 2$ if and only if the following hold true:
	\begin{align*}
		\sup\{\d_{1,\nu}\}_{\nu<\l} = r_1 &= \d_{1,\nu} \text{ for all }\nu\geq \nu_1,\\
		\sup\{\d_{j,\nu}\}_{\nu\geq\nu_{j-1}} = r_j &= \d_{j,\nu}\text{ for all }\nu\geq \nu_j, \text{ for all }j=2,\dotsc,i-1,\\
		\sup\{\d_{i,\nu}\}_{\nu\geq\nu_{i-1}} = r_i &> \d_{i,\nu}\text{ for all }\nu\geq\nu_{i-1}.
	\end{align*}
	Observe that $r_j\in\G_j$ for all $j<i$.
	\n (3) $\sup(E,\overline{v}) \in\overline{v}\overline{K}$ if and only if $\sup(E,\overline{v}) = (r_1,\dotsc,r_n)$ where $r_n\in\G_n$ and $r_n > \d_{n,\nu}$ for all $\nu\geq\nu_{n-1}$.
\end{Proposition}

The dual result when $E$ is a pds in $(K,v)$ is presented below:

\begin{Proposition}
	Assume that $E$ is a pds in $(K,v)$ with $\rk vK=n$ and fix an extension $\overline{v}$ of $v$ to $\overline{K}$. Identify $\overline{v}\overline{K}$ with $(\G_1\dirsum\dotsc\dirsum\G_n)_{\lex}$ where $\G_i$ are subgroups of $\RR$. Consider the identification (\ref{eqn identification}). Then the following statements hold true:
	\sn (1) $\inf(E,\overline{v}) = (r_1^\prime,\infty,\dotsc,\infty)$ if and only if $\inf\{\d_{1,\nu}\}_{\nu<\l} = r_1^\prime < \d_{1,\nu}$ for all $\nu<\l$.
	\n (2) $\inf(E,\overline{v}) = (r_1^\prime,\dotsc,r_i^\prime, \infty,\dotsc,\infty)$ for some $i\geq 2$ if and only if the following hold true:
	\begin{align*}
		\inf\{\d_{1,\nu}\}_{\nu<\l} = r_1^\prime &= \d_{1,\nu} \text{ for all }\nu\geq \nu_1, \\
		\inf\{\d_{j,\nu}\}_{\nu\geq\nu_{j-1}} = r_j^\prime &= \d_{j,\nu} \text{ for all }\nu\geq \nu_j, \text{ for all }j=2,\dotsc,n-1,\\
		\inf\{\d_{i,\nu}\}_{\nu\geq\nu_{i-1}} = r_i^\prime &< \d_{i,\nu}\text{ for all }\nu\geq\nu_{i-1}.
	\end{align*}
	Observe that $r_j^\prime\in\G_j$ for all $j<i$.
	\n (3) $\inf(E,\overline{v}) \in\overline{v}\overline{K}$ if and only if $\inf(E,\overline{v}) = (r_1^\prime,\dotsc,r_n^\prime)$ where $r_n^\prime\in\G_n$ and $r_n^\prime < \d_{n,\nu}$ for all $\nu\geq\nu_{n-1}$.
\end{Proposition}

Observe that if $E$ is a pcs of algebraic type with $\sup(E,\overline{v})\in\overline{v}\overline{K}$, then we arrive at a hollow node in Figure \ref{fig E pcs i level} at the $n$-th level and at the square node in each preceding level. Analogously, if $E$ is a pds with $\inf(E,\overline{v})\in\overline{v}\overline{K}$, then we arrive at a hollow node in Figure \ref{fig E pds i level} at the $n$-th level and at the square node in each preceding level. We thus have the following result:

\begin{Theorem}\label{Thm rank v_E}
	Assume that $(K,v)$ is a valued field of finite rank and fix an extension $\overline{v}$ of $v$ to $\overline{K}$. Further, assume that $E$ is a pms in $(K,v)$ of one of the following types:
	\sn (i) $E$ is a Cauchy sequnce,
	\n (ii) $E$ is a pcs of algebraic type with $\sup(E,\overline{v})\in\overline{v}\overline{K}$,
	\n (iii) $E$ diverges to infinity,
	\n (iv) $E$ is a pds with $\inf(E,\overline{v})\in\overline{v}\overline{K}$.\\
	Then,
	\[ \rk v_E K(X) = \rk vK +1. \]
	If $\rk vK=1$, then the above conditions are also necessary.
\end{Theorem}

It is clear from our preceding discussions that the above conditions are sufficient but not necessary for rank incrementation for valued fields of higher rank. An example illustrating this fact is furnished below.

\begin{Example}
	Let $(K,v)$ be a valued field with $vK  = (\G\dirsum\ZZ)_{\lex}$ where $\G$ is a nontrivial subgroup of $\QQ$. Fix an extension $\overline{v}$ of $v$ to $\overline{K}$. Then $\overline{v}\overline{K} = (\QQ\dirsum\QQ)_{\lex}$. Take an element $g\in\G$ and set $\d_i:= (g,i)$ for all $i\in\NN$. Take $z_i\in K$ such that $vz_i = \d_i$. By definition, $E:= \{z_i\}_{i\in\NN}$ is a pcs in $(K,v)$ with $0$ as a limit, hence $E$ is a pcs of algebraic type. Observe that $\sup (E,\overline{v}) = (g,\infty)\notin\overline{v}\overline{K}$. Further, we observe that we arrive at the square node in Figure \ref{fig E pcs base level} and we arrive at the top left hollow node in Figure \ref{fig E pcs i level} at the second level. As a consequence, $\rk v_E K(X) = \rk vK +1 = 3$.
\end{Example}

\end{document}